%% file: statpointmenger.tex
\documentclass[a4paper,10pt]{article}
\pdfoutput=1
\usepackage[utf8x]{inputenc}
\usepackage{amsmath,amsxtra}
\usepackage[varg]{txfonts}
\usepackage[unicode]{hyperref}
\usepackage{framed}
\usepackage[amsmath,hyperref,framed,thmmarks]{ntheorem}
\usepackage{mathrsfs}
\usepackage{graphicx}
\usepackage{bbm}
\usepackage{color}

\theoremseparator{.}
\makeatletter
\renewcommand{\newframedtheorem}[1]{%
\theoremprework{\framed\vspace{-1.5ex}}%
\theorempostwork{\vspace{-2ex}\endframed}%
\newtheorem@i{#1}%
}
\makeatother
\newframedtheorem{theorem}{Theorem}
\newframedtheorem{proposition}{Proposition}[section]
\newframedtheorem{lemma}[proposition]{Lemma}
\newframedtheorem{corollary}[proposition]{Corollary}
\theorembodyfont{\upshape}
\theoremsymbol{$\Diamond$}
\newtheorem{remark}[proposition]{Remark}
\makeatletter
\newtheoremstyle{proof}%
{\item[\theorem@headerfont\hskip\labelsep ##1\theorem@separator]}%
{\item[\ifx\@empty##3\else\theorem@headerfont\hskip \labelsep ##1\ of\ ##3\theorem@separator\fi]}
\makeatother
\theoremheaderfont{\itshape}
\theoremstyle{proof}
\theoremseparator{.}
\theoremsymbol{\ensuremath{\Box}}
\newtheorem{proof}{Proof}
\theoremheaderfont{\sffamily\bfseries}\theorembodyfont{\sffamily\upshape}
\theoremstyle{plain}
\theoremnumbering{alph}
\theoremsymbol{$\Diamond$}
\theoremprework{\small}
\theorempostwork{\normalsize}

\hypersetup{
breaklinks=true,
colorlinks=true,
linkcolor=blue,
citecolor=blue,
urlcolor=blue,
}

\allowdisplaybreaks

\def\bfseries{\fontseries \bfdefault \selectfont \boldmath}

\makeatletter
\def\@fnsymbol#1{\ensuremath{\ifcase#1\or *\or **\or {**}* \or {**}{**}\else\@ctrerr\fi}}
\makeatother

\newcommand{\abs}[1]{\ensuremath{\left|#1\right|}}
\newcommand{\br}[1]{\ensuremath{\left(#1\right)}}
\newcommand{\bv}[1]{\ensuremath{\left[#1\right]}}
\renewcommand{\d}{\ensuremath{\,\mathrm{d}}}

\newcommand{\lnorm}[2][2]{\ensuremath{\left\|#2\right\|_{L^{{\scriptsize#1}}}}}
\newcommand{\norm}[1]{\ensuremath{\left\|#1\right\|}}
\newcommand{\seminorm}[1]{\ensuremath{\left[#1\right]}}
\newcommand{\seminormv}[1]{\ensuremath{\left\llbracket#1\right\rrbracket}}

\newcommand{\seqn}[2][k]{\ensuremath{\br{{#2}_{#1}}_{#1\in\N}}}

\newcommand{\set}[1]{\ensuremath{\left\{#1\right\}}}
\newcommand{\sets}[2]{\ensuremath{\left\{\left.#1\,\right|#2\right\}}}
\newcommand{\sett}[2]{\ensuremath{\left\{#1\left|\,#2\right.\right\}}}
\renewcommand{\sp}[1]{\ensuremath{\left\langle#1\right\rangle}}

\numberwithin{equation}{section}

\DeclareFontFamily{U}{matha}{\hyphenchar\font45}
\DeclareFontShape{U}{matha}{m}{n}{
      <5> <6> <7> <8> <9> <10> gen * matha
      <10.95> matha10 <12> <14.4> <17.28> <20.74> <24.88> matha12
      }{}
\DeclareSymbolFont{matha}{U}{matha}{m}{n}
\DeclareFontFamily{U}{mathx}{\hyphenchar\font45}
\DeclareFontShape{U}{mathx}{m}{n}{
      <5> <6> <7> <8> <9> <10>
      <10.95> <12> <14.4> <17.28> <20.74> <24.88>
      mathx10
      }{}
\DeclareSymbolFont{mathx}{U}{mathx}{m}{n}
\DeclareMathSymbol{\oast}{2}{matha}{"66}
\DeclareMathSymbol{\bigoast}{1}{mathx}{"C6}

\renewcommand{\a}{\ensuremath{\alpha}}
\renewcommand{\angle}{\sphericalangle}

\renewcommand{\C}{\ensuremath{\mathbb{C}}}

\newcommand{\Cia}[1][1]{\ensuremath{C_{\mathrm{ia}}^{#1}}}

\newcommand{\D}[3][\g]{{#1}(u+#2)-{#1}(u+#3)}

\newcommand{\dg}{{\gamma}'}

\renewcommand{\dh}{{h}'}

\newcommand{\E}[1][p,q]{\ensuremath{\mathrm{intM}^{(#1)}}}

\newcommand{\eps}{\ensuremath{\varepsilon}}

\newcommand{\fracabs}[1]{\frac{#1}{\abs{#1}}}
\newcommand{\g}{\gamma}
\newcommand{\gp}{g^{p}}

\renewcommand{\G}[1][p]{\ensuremath{\mathcal G^{(#1)}}}

\newcommand{\N}{\ensuremath{\mathbb{N}}}

\newcommand{\Q}[1][p]{Q^{(#1)}}

\newcommand{\R}{\ensuremath{\mathbb{R}}}
\renewcommand{\Re}[1][\eps]{R_{#1}^{(p)}}
\newcommand{\refeq}[2][=]{\ensuremath{\stackrel{\text{\makebox[0ex][c]{\eqref{eq:#2}}}}{#1}}}
\renewcommand{\rho}{\ensuremath{\varrho}}
\newcommand{\RR}[1][p,q]{R^{(#1)}}
\newcommand{\rzd}{\ensuremath{(\R/\Z,\R^n)}}
\newcommand{\s}{\ensuremath{\sigma}}
\DeclareMathOperator*{\sign}{sign}

\newcommand{\T}[3][\g]{\triangle_{#2,#3}{#1}}

\newcommand{\tg}{\tilde{\gamma}}
\renewcommand{\th}{\ensuremath{\vartheta}}

\newcommand{\tint}{\iiint_{\R/\Z\times D}}
\newcommand{\vth}{\ensuremath{\vartheta}}
\newcommand{\W}[1][(3p-2)/q-1,q]{\ensuremath{W^{\scriptstyle #1}}}

\newcommand{\Wia}[1][(3p-2)/q-1,q]{\ensuremath{W_{\mathrm{ia}}^{#1}}}
\newcommand{\Z}{\ensuremath{\mathbb{Z}}}

\title{Towards a regularity theory \\ for integral Menger curvature}
\author{%
 Simon Blatt\thanks{Workgroup Applied Analysis, Karlsruhe Institute of Technology, Kaiserstra{\ss}e 89 -- 93, 76133 Karlsruhe, Germany, \url{simon.blatt@kit.edu}}
 \and
 Philipp Reiter\thanks{Fakult\"at f\"ur Mathematik, Universit\"at Duisburg-Essen, Forsthausweg 2, 47057 Duisburg, Germany, \url{philipp.reiter@uni-due.de}}}

\begin{document}
\maketitle
\begin{abstract}
We generalize the notion of
integral Menger curvature introduced by Gonzalez and Maddocks~\cite{gonz-madd} by decoupling 
the powers in the integrand.
This leads to a new two-parameter family
of knot energies $\E$.

We classify finite-energy curves in terms of Sobolev-Slobodecki{\u\i} spaces.
Moreover, restricting to the range
of parameters leading to a sub-critical Euler-Lagrange equation,
we prove existence of minimizers within any knot class via a
uniform bi-Lipschitz bound.
Consequemtly, $\E$ is a knot energy in the sense of O'Hara.

Restricting to the non-degenerate sub-critical case,
a suitable decomposition of the first variation allows to
establish a bootstrapping argument that leads to $C^{\infty}$-smoothness
of critical points.

\end{abstract}
\tableofcontents

\newcommand{\secA}{Classification of finite-energy curves}
\newcommand{\secB}{Existence of minimizers within knot classes}
\newcommand{\secC}{Differentiability}
\newcommand{\secD}{Regularity of stationary points}

\input{intro.tex}
\input{energy-space}
\input{regularity}

\begin{appendix}
 \input{fractional}
\end{appendix}


\input{statpointmenger.bbl}
\end{document}

%% file: intro.tex

\section*{Introduction}
\addcontentsline{toc}{section}{Introduction}

\emph{Imagine a closed curve in the Euclidean space.
Each triple of distinct points on the curve uniquely defines
its {circumcircle} that passes through these three points.
It degenerates to a line if and only if the points are collinear.
The reciprocal of the circumcircle radius can be seen
as some kind of approximate curvature.
How much information on shape and regularity of the curve
can be drawn from the $L^{p}$-norm
of the latter quantity?}

Motivated from applications in microbiology,
Gonzalez and Maddocks~\cite{gonz-madd} investigated
this question for the case $p=\infty$.
They were in search of a notion for the \emph{thickness}
of an embedded curve that, in contrast to
other approaches, e.g.\@ Litherland et al.~\cite{lsdr}, does not require
initial regularity of the respective curves.

Thickness is influenced by both local and global properties of
a curve and is additionally related to the regularity of the curve.
In fact, the thickness of an arc-length parametrized curve
is finite if and only if it is embedded and
has a Lipschitz continuous tangent, i.e., it is $C^{1,1}$,
see Gonzalez et al.~\cite{gmsm}.
Consequently, any curve of finite thickness
parametrized by arc-length
is bi-Lipschitz continuous with a bi-Lipschitz
constant only depending on its thickness.

The latter is particularly interesting in the
context of applications. Instead of
trying to immediately determine the knot type
of a given possibly quite entangled curve, one could first ``simplify'' it
in order to obtain a nicely shaped curve,
having large distances between distant strands.
Such a deformation process could be defined by the gradient flow of a
suitable functional which should prevent the curve
from leaving the ambient knot class,
preserving the bi-Lipschitz property.

This idea was formalized into the concept of
\emph{knot energies} by O'Hara~\cite[Def.~1.1]{oha:en-kn}.
A functional on a given space of knots is
called a {knot energy} if
it is bounded below and self-repulsive,
i.e., it blows up on sequences of embedded curves converging to
a non-embedded limit curve.

Among other functionals Gonzalez and Maddocks~\cite{gonz-madd} also proposed
to investigate the functional
\[ \mathscr M_{p}(\g) := \iiint\limits_{(\R/\Z)^{3}}
 \frac{\abs{\dg(u_{1})}\abs{\dg(u_{2})}\abs{\dg(u_{3})}}{R(\g(u_{1}),\g(u_{2}),\g(u_{3}))^{p}} \d u_{1}\d u_{2}\d u_{3}, \qquad p\in(0,\infty), \]
which is called \emph{integral Menger\footnote{%
Named after \textsc{Karl Menger}, 1902 -- 1985, US-Austrian mathematician,
who used the circumcircle for generalizing geometric concepts
to general metric spaces~\cite{menger}.
He worked on many fields including
distance geometry, dimension theory, graph theory.
Menger was a student of Hahn in Vienna
where he received a professorship in 1927.
Being member of the Vienna Circle, he was also
interested in philosophy and social science.
After emigrating to the USA in 1937,
he obtained a position at Notre Dame, later at Chicago.
See Kass~\cite{kass:menger} for further reading.
}%
curvature}.
Here $\g:\R/\Z\to\R^{n}$ is an absolutely continuous curve and
$R(x,y,z)$ denotes the circumcircle 
of the three points $x,y,z\in\R^{n}$ given by 
\begin{equation}
 R(x,y,z):= \frac{{\abs{y-z}\abs{y-x}\abs{z-x}}}{2\abs{\br{y-x}\wedge\br{z-x}}}=\frac{\abs{y-z}}{2\sin\angle\br{y-x,z-x}}.
\end{equation}

The functionals $\mathscr M_{p}$ have been investigated by
Strzelecki, Szuma\'nska and von der Mosel in~\cite{SM5}
wherein further references can be found.
Their results cover the case $p>3$ where $\mathscr M_{p}$
is known to be a knot energy.
Especially they have been able to show that
finite energy of an arc-length parametrized curve
implies $C^{1,1-3/p}$-regularity
and its image is $C^{1}$-homeomorphic to the circle.
The regularity statement has been sharpened in~\cite{blatt:imcc}.

The element $\mathscr M_{2}$ is referred to as \emph{total Menger curvature}.
Interestingly, it plays an important r\^ole in complex analysis,
more precisely in the proof of
\emph{Vitushkin’s conjecture},
a partial solution to \emph{Painlev\'e's problem}
which asks to determine \emph{removable} sets.
These are compact sets $K\subset\C$
such that for any open $U\subset\C$ containing $K$
and for any bounded analytic function $U\setminus K\to\C$,
the latter can be extended to an analytic function on $U$.
Vitushkin conjectured that a compact set $K$ with positive finite
one-dimensional Hausdorff measure is
removable if and only if it is \emph{purely unrectifiable},
i.e.\@ it intersects every rectifiable curve in a set of measure zero.

A central result in this context is the
the curvature theorem of David and
L\'eger~\cite{leger} stating that one-dimensional Borel sets in $\C$
with finite total Menger curvature are 1-rectifiable.
Mel{$'$}nikov and Verdera~\cite{melnikov,melnikov-verdera,verdera}
discovered a connection between
$L^{2}$-boundedness of the Cauchy integral operator on
Lipschitz graphs and the Menger curvature.
For further details regarding
Vitushkin’s conjecture for removable sets
we refer to
Dudziak's monography~\cite{dudziak} and references therein.

Menger curvature for higher-dimensional objects has been discussed
in~\cite{kola1,kolaSM,kolaSz,Bkola,kola2,scholtes}.
Further information on the context of the integral Menger curvature
within the field of geometric knot theory and geometric curvature energies
can be found in the recent surveys by Strzelecki and von der Mosel~\cite{SM10,SM11}.

In this article we make a first step towards the
regularity theory of stationary points of {integral Menger curvature}.
Regularity theory for minimizers of certain knot energies
has been developed in~\cite{oha:fam-en,fhw,he:elghf,reiter:rkepdc,reiter:rtme,blatt-reiter2,blatt-reiter3}.
A summary is given in~\cite{blatt-reiter:proc,blatt-reiter:mbmb}.

Unfortunately
the Euler-Lagrange operator of $\mathscr M_{p}$ is not only non-local
but also degenerate.
In order to produce non-degenerate energies,
we embed this family into the two-parameter family of generalized integral Menger curvature
\begin{equation}\label{eq:intM}
 \E(\g):=\iiint\limits_{(\R/\Z)^{3}}
 \frac{\abs{\dg(u_{1})}\abs{\dg(u_{2})}\abs{\dg(u_{3})}}{\RR(\g(u_{1}),\g(u_{2}),\g(u_{3}))}
 {\d u_{1}\d u_{2}\d u_{3}}, \qquad p,q>0,
\end{equation}
where
\begin{equation}
 \RR(x,y,z) :=
 \frac{\br{\abs{y-z}\abs{y-x}\abs{z-x}}^{p}}{\abs{\br{y-x}\wedge\br{z-x}}^{q}}
 = \frac{\abs{y-z}^{p}\abs{y-x}^{p-q}\abs{z-x}^{p-q}}{\sin\angle\br{y-x,z-x}^{q}},
 \qquad x,y,z\in\R^{n}.
\end{equation}
Note that the function $\RR$ is symmetric in all components.
Of course, $\mathscr M_{p} = 2^{p}\E[p,p]$.

The elements of this family are
knot energies 
under certain conditions only. More precisely, we will see
in Remark~\ref{rem:non-rep} that
they are punishing self-intersections if and only if
\begin{equation}\label{eq:self-avoiding}
 p\ge\tfrac23q+1.
\end{equation}
On the other hand, these energies can only be finite on closed curves iff
\begin{equation}\label{eq:singular}
 p<q+\tfrac23,
\end{equation}
see Remark~\ref{rem:singular}.
The \emph{sub-critical range},
i.e.\@ the range of parameters that lead to a {sub}-critical Euler-Lagrange equation,\footnote{In contrast, the corresponding range is called \emph{super}-critical by Strzelecki et al.~\cite{SM9}
as it lies above the respective critical value for which the
energy is scale-invariant.}
is given by
\begin{equation}\label{eq:sub-critical}
 p\in\br{\tfrac23q+1,q+\tfrac23} \qquad (q>1),
\end{equation}
and its non-degenerate part (leading to a non-degenerate equation) is
\begin{equation}\label{eq:non-degenerate-sub-critical}
 p\in\br{\tfrac73,\tfrac83}, \qquad q=2.
\end{equation}
These areas are visualized in Figure~\ref{fig:range}.

\begin{figure}[h]
 \centering
 \includegraphics{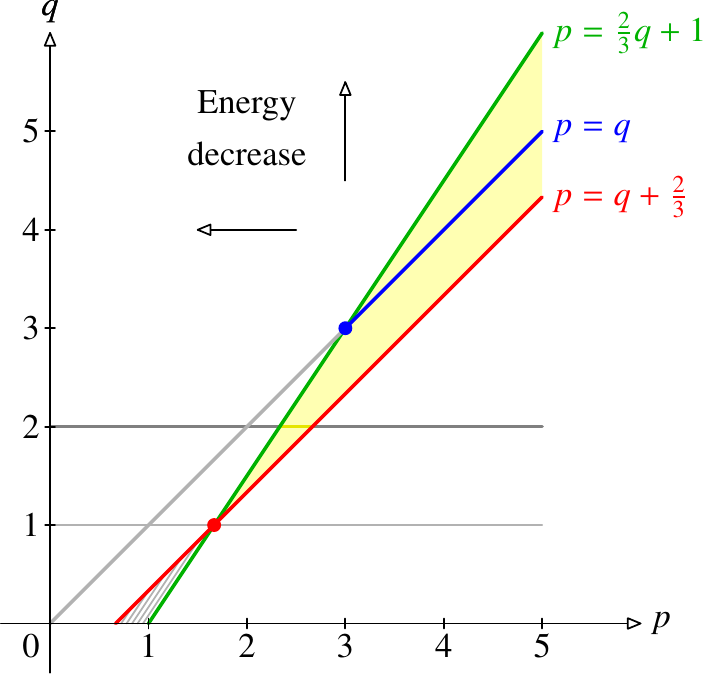}
 \caption{The range of $\E$.
 Above the green line, the integrand is not sufficiently singular to
penalize self-intersections, thus $\E$ is not a knot energy.
On the other hand, below the red line, for $q>1$, the integrand is so singular, that the integral is
either equal to zero or infinite, so
 there are no finite-energy $C^{1}$-knots at all.
The hatched area reveals the strange behavior that there are no finite-energy
$C^{3}$-knots while it takes finite values on polygons.}
 \label{fig:range}
\end{figure}

In~\cite{blatt:imcc}, a characterization of curves with finite $\mathscr M_{p}$ energy was given in terms of function spaces. Using this technique we infer

\begin{theorem}[\secA]\label{thm:energy-space}\hfill\\
 Consider the sub-critical case~\eqref{eq:sub-critical} and
 let $\g \in C^1(\R/\Z,\R^{n})$ be an injective curve parametrized by 
 arc-length.
 Then $\E(\g)<\infty$ if and only if~$\g \in\W$. Moreover,
 one then has, for constants $C,\beta>0$ depending on $p,q$ only,
 \begin{equation}\label{eq:energy-space}
  \norm\g_{\W} \leq C\br{\E(\g) + \E(\g)^\beta}.
 \end{equation}
\end{theorem}

We will use the last theorem to show

\begin{theorem}[\secB]\label{thm:existence}\hfill\\
 In the sub-critical case~\eqref{eq:sub-critical},
 there is a minimizer of 
 $\E$ among all injective, regular curves 
 $\g \in C^1(\R / \Z, \R^n)$ in any knot class.
\end{theorem}

To shorten notation we use the abbreviation
\begin{equation}\label{eq:short-notation}
 \T[\bullet]{v}{w}:=\D[\bullet]vw
\end{equation}
throughout this paper.
Furthermore, we sometimes omit the argument of a function if it is precisely the variable $u$, i.e.\@
$\g=\g(u)$ etc.

The first variation of $\mathscr M_{p}$, $p\ge2$,
has been derived by Hermes~\cite[Thm.~2.33, Rem.~2.35]{hermes}.
Here we use a different approach to prove

\begin{theorem}[\secC]\label{thm:first-var}
 In the sub-critical case~\eqref{eq:sub-critical}
 the functional $\E$ is $C^1$ on the subspace of all regular embedded
 $\W$-curves.

 For any arc-length parameterized embedded $\g\in\W\rzd$
 and $h\in\W\rzd$, the first variation of~$\E$ at~$\g$ in direction~$h$
 amounts to
 \begin{equation}\label{eq:dE-intro}
 \begin{split}
   &\delta\E(\g,h) \\
   &= \iiint\limits_{(\mathbb R / \mathbb Z)^3} \Bigg\{
 2q\frac{\abs{\T v0\wedge\T w0}^{q-2}}{\br{\abs{\T vw}\abs{\T v0}\abs{\T w0}}^p}
 \cdot{\sp{\T v0\wedge\T w0,\T v0\wedge\T[h] w0}} \\
 &\qquad\qquad\qquad{}-3p\frac{\abs{\T v0\wedge\T w0}^{q}}{\br{\abs{\T vw}\abs{\T v0}\abs{\T w0}}^p}\cdot\frac{\sp{\T vw,\T[h] vw}}{\abs{\T vw}^{2}}\\
 &\qquad\qquad\qquad{}+3\frac{\abs{\T v0\wedge\T w0}^{q}}{\br{\abs{\T vw}\abs{\T v0}\abs{\T w0}}^p}\cdot\sp{\dg,\dh}
\Bigg\}\d w\d v\d u.
 \end{split}
 \end{equation}
\end{theorem}

Using this formula, we will see that stationary points of 
the energies $\E [p,2]$ restricted to fixed length satisfy some kind of non-local uniformly elliptic 
pseudo-differential equation. If furthermore  $p \in\br{\tfrac73,\tfrac83}$,
the non-linearity turns out to be sub-critical and  we can finally use the Euler-Lagrange equation
to prove the following main result of this article:

\begin{theorem}[\secD]\label{thm:smooth}\hfill\\
 For $p\in(\tfrac73,\tfrac83)$, let $\g\in\W[3p/2-2,2]\rzd$ be a stationary point of $\E[p,2]$ with respect to fixed length, injective and parametrized by arc-length.
 Then $\g\in C^{\infty}$.
\end{theorem}

In a sense this concludes our study of the non-degenerate, 
subcritical cases of the most prominent knot energies
for curves. Regularity theory for the non-degenerate sub-critical case
has already been performed for O'Hara's energies~\cite{blatt-reiter2}
and for the generalized tangent-point energies~\cite{blatt-reiter3}.
The treatment of the critical case however
turns out to be far more involved and has yet only be done for 
O'Hara's knot energies~\cite{blatt-reiter-schikorra}.

We briefly introduce \emph{Sobolev-Slobodecki{\u\i} spaces} in the form
we will use them in this text.
Let $f\in W^{1,1}\rzd$. For $s\in(0,1)$ and $\rho\in[1,\infty)$ we define the seminorm
\begin{equation}\label{eq:Wsemi}
 \seminorm{f}_{\W[1+s,\rho]} := \br{\int_{{\R/\Z}}\int_{-1/2}^{1/2}
 \frac{\abs{f'(u+w)-f'(u)}^\rho}{\abs w^{1+\rho s}} \d w\d u}^{1/\rho}.
\end{equation}
On $W^{1,\rho}$ this seminorm is equivalent to
\begin{equation}\label{eq:Wsemi'}
 \seminormv{f}_{\W[1+s,\rho]} := \br{\int_{{\R/\Z}}\int_{-1/4}^{1/4}
 \frac{\abs{f(u+w)-2f(u)+f(u-w)}^\rho}{\abs w^{1+\rho(1+s)}} \d w\d u}^{1/\rho},
\end{equation}
see Appendix~\ref{sect:equiv}.

Now let $W^{k,\rho}\rzd$, $k\in\N$, denote the usual Sobolev space
(recall $W^{0,\rho}:=L^\rho$) and
\[ W^{k+s,\rho}\rzd := \sett{f\in W^{k,\rho}\rzd}{\norm f_{\W[k+s,\rho]}<\infty} \]
which we equip, depending on the situation, either with the norm
$\norm f_{\W[k,\rho]} + \seminorm{f^{(k-1)}}_{\W[1+s,\rho]}$
or with $\norm f_{\W[k,\rho]} + \seminormv{f^{(k-1)}}_{\W[1+s,\rho]}$
respectively.

Without further notice we will frequently use the embedding
\begin{equation}\label{eq:Wembed}
 \W[k+s,\rho]\rzd\hookrightarrow C^{k,s-1/\rho}\rzd,
 \qquad \rho\in(1,\infty), \quad s\in(\rho^{-1},1).
\end{equation}
We will denote by $\Cia[]$ resp.\@ $\Wia[]$
\underline injective (embedded) curves parametrized by \underline arc-length.
As usual, a curve is said to be \emph{regular} if there is some $c>0$ such that $\abs\dg\ge c$ a.e.
Constants may change from line to line.

%
%
%

%% file: energy-space.tex
\section{\secA}

Before we begin the discussion of the first variation, let us rewrite the integral Menger curvature using
the symmetry and a suitable covering of the domain of integration $(\R/\Z)^3$ by domains, on which it is 
easier to estimate the terms that will appear.
The general idea here is quite similar
to~\cite{blatt:imcc} and Hermes~\cite{hermes}, but we will 
show that it is actually enough to integrate over a certain subdomain of $(\mathbb R / \mathbb Z)^3$.

To this end we define the range of integration
\begin{equation}\label{eq:deco}
 D:=\sets{(v,w) \in (-\tfrac12,0) \times (0,\tfrac12)}{w\leq 1+2v, v \geq -1 + 2w},
\end{equation}
which is depicted in Figure~\ref{fig:domain}.

\begin{figure}[h]
 \centering
 \includegraphics{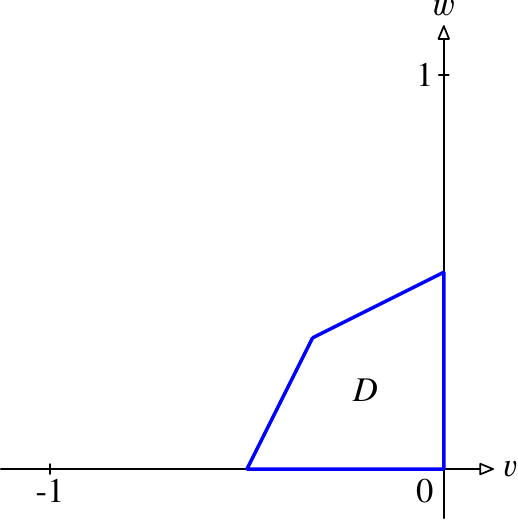}
  \caption{The range of integration $D$}
 \label{fig:domain}
\end{figure}

\begin{lemma}[Domain decomposition]\label{lem:domdec}
 Let $f\in L^{1}\br{\br{\R/\Z}^{3}}$ be symmetric in all components,
 i.e.\@ $f=f\circ\s$ for all permutations $\s\in\mathfrak S_{3}$.
 Then
 \[ \iiint_{(\R/\Z)^{3}} f(u_{1},u_{2},u_{3})\d u_{1}\d u_{2}\d u_{3} = 6\iiint_{\R/\Z\times D} f(u,u+v,u+w)\d w\d v\d u. \]
\end{lemma}

\begin{proof}
 Let $P_{\s}\in\R^{3\times 3}$ denote the permutation matrix corresponding to $\s\in\mathfrak S_{3}$. 
 We first show that the images
 $\sets{P_{\s}\br{\R/\Z\times D}}{\s\in\mathfrak S_{3}}$
 cover $(\R/\Z)^{3}$.

 Consider $(u_{1},u_{2},u_{3})\in\br{\R/\Z}^{3}$. Then after a suitable permutation we can assume that
 \begin{equation*}
  d_{\mathbb R / \mathbb Z}(u_1,u_3) = \max\br{d_{\mathbb R / \mathbb Z}(u_1,u_2),d_{\mathbb R / \mathbb Z}(u_2,u_3),
  d_{\mathbb R / \mathbb Z}(u_1, u_3) }
 \end{equation*}
 where $d_{\mathbb R / \mathbb Z}(x,y) = \min\{|x-y-k|: k \in \mathbb Z\}
 \in[0,\frac12]$ denotes the
 distance in $\mathbb R / \mathbb Z$.
 Hence, interchanging $u_{1}$ and $u_{3}$ if necessary,
 there are $(v,w) \in \bv{-\frac12,0}\times\bv{0,\frac12}$ with
 $$
   u_1 = u_2 + v,\qquad u_3 = u_2 + w,
 $$
 and 
 \begin{align*} \max\br{-v,w} &=\max\br{d_{\R/\Z}(u_{1},u_{2}),d_{\R/\Z}(u_{2},u_{3})}
    \\ &\le d_{\R/\Z}(u_{1},u_{3}) = \min\br{w-v,1-(w-v)}\le1-w+v, \end{align*}
 so $(v,w) \in\overline D$.
 Since furthermore 
 \begin{equation*}
\#\mathfrak S_3 \cdot \abs{\R/\Z\times D}=6 \cdot \tfrac16 = \abs{(\mathbb R / \mathbb Z)^3}
 \end{equation*}
 the sets $\sets{P_{\s}\br{\R/\Z\times D}}{\s\in\mathfrak S_{3}}$
 form up to sets of measure zero a disjoint partition of $(\R/\Z)^{3}$.
\end{proof}

Following Lemma~\ref{lem:domdec} we derive using~\eqref{eq:short-notation}
\begin{equation}\label{eq:intM-sn''}
 \E(\g)=6\iiint_{\R/\Z\times D}
 \frac{\abs{\T v0\wedge\T 0w}^{q}}{\br{\abs{\T v0}\abs{\T 0w}\abs{\T vw}}^p}
 {\abs{\dg(u)}\abs{\dg(u+v)}\abs{\dg(u+w)}}
 {\d w\d v\d u}.
\end{equation}

\begin{proof}[Theorem~\ref{thm:energy-space}]
Recall that any embedded $\Wia[1+s,\rho]$-curve, $s>\tfrac1\rho$, is bi-Lipschitz continuous~\cite[Lemma~2.1]{blatt:bre}.
We obtain, using $\abs{a\wedge b}=\abs{a\wedge(a\pm b)}\le\abs a\abs{a\pm b}$ for $a,b\in\R^{3}$,
\begin{align*}
 \E(\g) &\le C \iiint_{\R/\Z\times D}\frac{\abs{\int_0^1\dg(u+\theta_{1} v)\d\theta_{1}\wedge
 \int_0^1\dg(u+\theta_{2} w)\d\theta_{2}}^{q}}{\abs v^{p-q}\abs w^{p-q}\abs{v-w}^{p}} \d w\d v\d u\\
 &\le C \tint\frac{\int_0^1\abs{\dg(u+\theta v)-\dg(u+\theta w)}^{q}\d\theta}{\abs v^{p-q}\abs w^{p-q}\abs{v-w}^{p}} \d w\d v\d u\\
 &\le C \int_0^1\tint\frac{\abs{\dg(u+\theta (v-w))-\dg(u)}^{q}}{\abs v^{p-q}\abs w^{p-q}\abs{v-w}^{p}} \d w\d v\d u\d\theta
\end{align*}
where $C$ depends on $p,q$ and $\g$.
Substituting
$\Phi:(v,w)\mapsto (t,\tilde w) := \br{\frac v{v-w},\theta(v-w)}$,
$\abs{\det D\Phi(v,w)} = \frac\theta{\abs{v-w}}$,
$\Phi(D)\subset[0,1]\times[-1,0]$,
we arrive at
\begin{equation}\label{eq:intM<seminorm}
\begin{split}
 &\E(\g) \\
 &\le C \int_{0}^{1}\theta^{-1}\int_{\R/\Z}\int_{0}^{1}\int_{-1}^{0}
 \frac{\abs{\dg(u+\tilde w)-\dg(u)}^{q}}{\abs {t\frac{\tilde w}\theta}^{p-q}\abs{(t-1)\frac{\tilde w}\theta}^{p-q}\abs{\frac{\tilde w}\theta}^{p-1}} \d\tilde w\d t\d u\d\theta \\
 &\le C \int_{0}^{1}{\theta^{3p-2q-2}}\d\theta
 \int_{0}^{1}\frac{\d t}{\abs{t(1-t)}^{p-q}}
 \int_{\R/\Z}\int_{-1}^{0}
 \frac{\abs{\dg(u+\tilde w)-\dg(u)}^{q}}{\abs {\tilde w}^{3p-2q-1}} \d\tilde w\d u \\
 &\le C\br{\seminorm{\g}_{\W}^{q}+\lnorm[\infty]\dg^{q}}
 \le C\norm{\g}_{\W}^{q}.
\end{split}
\end{equation}
For the other implication, we first derive for given vectors
$a,b\in\R^{n}$, $\abs a=\abs b=1$, $\sp{a,b}\ge0$,
\begin{equation}\label{eq:a-wedge-b}
 \abs{a\wedge b}^{2} = \abs a^{2}\abs b^{2}-\sp{a,b}^{2}
   \ge1-\sp{a,b}=\tfrac12\abs{a-b}^{2}.
\end{equation}
By uniform continuity of $\dg$
we may choose $\delta=\delta(\g)\in(0,\tfrac12)$ such that
\begin{equation}
 \abs{\dg(u+v)-\dg(u+w)} \le \tfrac1{10}
 \qquad \text{for all } u\in\R/\Z, \quad v,w\in[-\delta,\delta].
\end{equation}
In fact, we may choose $\delta$ to be maximal, i.e.\@ we
assume that there are $\tilde u\in\R/\Z$, $\tilde v,\tilde w\in[-\delta,\delta]$
with
\begin{equation}\label{eq:delta=max}
 \abs{\dg(\tilde u+\tilde v)-\dg(\tilde u+\tilde w)} = \tfrac1{10}.
\end{equation}
We fix $u_{0}\in\R/\Z$. As $\g\in C^{1}\rzd$ we may apply a suitable
translation and rotation of the ambient space~$\R^{n}$ such that
$\g(u_{0})=0$ and there is a function $f\in C^{1}(\R,\R^{n-1})$
with $\lnorm[\infty]f\le1$ and $f(0)=0$ such that
$\tilde\g(u):=(u,f(u))$ satisfies $\tilde\g(B_{2\delta}(0))\subset \g(\R/\Z)$.
Then
\begin{equation}\label{eq:tilde-bilip}
 \tfrac12\abs{\T[\tilde\g]v0}\le\abs{v}\le\abs{\T[\tilde\g]v0}.
\end{equation}

Arc-length parametrization of $\g$ gives
\renewcommand{\T}[3][\tg]{\triangle_{#2,#3}{#1}}
\begin{equation}\label{eq:tilde-g}
\begin{split}
 &\E(\tilde\g) \\
 &\ge c\int_{u_0-\delta}^{u_0+\delta}\int_{-\delta}^\delta\int_{-\delta}^\delta
 \frac{\abs{\fracabs{\T v0}\wedge\fracabs{\T 0w}}^{q}}{{\abs{\T v0}^{p-q}\abs{\T 0w}^{p-q}\abs{\T vw}^p}}
 {\d w\d v\d u} \\
 &\refeq[\ge]{a-wedge-b} c\int_{u_0-\delta}^{u_0+\delta}\int_{-\delta}^\delta\int_{-\delta}^\delta
 \frac{\abs{\sign v\fracabs{\T v0}+\sign w\fracabs{\T 0w}}^{q}}{{\abs{\T v0}^{p-q}\abs{\T 0w}^{p-q}\abs{\T vw}^p}}
 {\d w\d v\d u} \\
 &\refeq[\ge]{tilde-bilip} c\int_{u_0-\delta}^{u_0+\delta}\int_{-\delta}^\delta\int_{-\delta}^\delta
 \frac{\abs{\sign v\fracabs{\T v0}+\sign w\fracabs{\T 0w}}^{q}}{{\abs{v}^{p-q}\abs{w}^{p-q}\abs{v-w}^p}}
 {\d w\d v\d u} \\
 &\ge c\int_{u_0-\delta}^{u_0+\delta}\int_{-\delta}^\delta\int_{-\delta}^\delta
 \br{\frac{\abs{\sign v\fracabs{\T v0}+\sign w\fracabs{\T 0w}}^{q}}{{\abs{v}^{p-q}\abs{w}^{p-q}\abs{v-w}^p}}
 + \frac{\abs{-\sign v\fracabs{\T {-v}0}+\sign w\fracabs{\T 0w}}^{q}}{{\abs{v}^{p-q}\abs{w}^{p-q}\abs{v+w}^p}}}
 {\d w\d v\d u} \\
 &\ge c\int_{u_0-\delta}^{u_0+\delta}\int_{-\delta}^{\delta}\int_{-\abs v}^{\abs v}
 \br{\frac{\abs{\sign v\fracabs{\T v0}+\sign w\fracabs{\T 0w}}^{q}}{{\abs{v}^{p-q}\abs{w}^{p-q}\abs{v+w}^p}}
 +\frac{\abs{-\sign v\fracabs{\T {-v}0}+\sign w\fracabs{\T 0w}}^{q}}{{\abs{v}^{p-q}\abs{w}^{p-q}\abs{v-w}^p}}}
 {\d w\d v\d u} \\
 &\ge c\int_{u_0-\delta}^{u_0+\delta}\int_{-\delta}^{\delta}\int_{-\abs v}^{\abs v}
 \frac{\abs{\sign v\fracabs{\T v0}+\sign w\fracabs{\T 0w}}^{q}
 +\abs{-\sign v\fracabs{\T {-v}0}+\sign w\fracabs{\T 0w}}^{q}}{{\abs{v}^{3p-2q}}}
 {\d w\d v\d u} \\
 &\ge c\int_{u_0-\delta}^{u_0+\delta}\int_{-\delta}^{\delta}
 \frac{\abs{\fracabs{\T v0}+{\fracabs{\T {-v}0}}}^{q}}{{\abs{v}^{3p-2q-1}}}
 {\d v\d u}
%
%
\end{split}
\end{equation}
where $c>0$ only depends on $p$ and $q$.
The last line in~\eqref{eq:tilde-g} is bounded below by
\begin{align*}
 &c\int_{u_0-\delta}^{u_0+\delta}\int_{-\delta}^{\delta}
 \frac{\abs{{\T v0}+{{\T {-v}0}}}^{q}}{{\abs{v}^{3p-q-1}}}
 {\d v\d u}
 -C\int_{u_0-\delta}^{u_0+\delta}\int_{-\delta}^{\delta}
 \frac{\abs{{\T{-v}0}}^{q}\abs{\frac1{\abs{\T v0}}-\frac1{\abs{\T {-v}0}}}^{q}}{{\abs{v}^{3p-2q-1}}}
 {\d v\d u} \\
 &\refeq[\ge]{tilde-bilip} c{
 \int_{u_{0}-\delta}^{u_{0}+\delta}\int_{-\delta}^{\delta}
 \frac{\abs{\tg(u+v)-2\tg(u)+\tg(u-v)}^{q}}{\abs v^{3p-q-1}}\d v\d u
 }
 -C\int_{u_0-\delta}^{u_0+\delta}\int_{-\delta}^{\delta}
 \frac{\abs{\frac v{\abs{\T v0}}-\frac{v}{\abs{\T {-v}0}}}^{q}}{{\abs{v}^{3p-2q-1}}}
 {\d v\d u}.
\end{align*}
By
\[ \abs{\frac v{\abs{\T v0}}-\frac{v}{\abs{\T {-v}0}}}
   \le \abs{\frac{(v,\T[f]v0)}{\abs{\T v0}}+\frac{(-v,\T[f]{-v}0)}{\abs{\T {-v}0}}}
   = \abs{\fracabs{\T v0}+\fracabs{\T {-v}0}} \]
we may use~\eqref{eq:tilde-g} to absorb the last term
which finally leads to
\[ \E(\g)\ge c{
 \int_{u_{0}-\delta}^{u_{0}+\delta}\int_{-\delta}^{\delta}
 \frac{\abs{\tg(u+v)-2\tg(u)+\tg(u-v)}^{q}}{\abs v^{3p-q-1}}\d v\d u
 }. \]
Since reparametrization to arc-length preserves regularity, we arrive at
\begin{equation}\label{eq:delta-est}
 \E(\g)\ge c{
 \int_{u_{0}-\delta}^{u_{0}+\delta}\int_{-\delta}^{\delta}
 \frac{\abs{\g(u+v)-2\g(u)+\g(u-v)}^{q}}{\abs v^{3p-q-1}}\d v\d u
 }.
\end{equation}
As $u_{0}$ was chosen arbitrarily, we obtain
\begin{equation}\label{eq:semi-intM}
 {\seminormv{\g}}_{\W[(3p-2)/q-1,q]}^{q}
 \le C\br{\E(\g)+\lnorm[\infty]{\dg}^{q}\delta^{-3p+2q+2}}
\end{equation}
uniformly on $\R/\Z$.
Since the exponent $-3p+2q+2$ is negative,
we have to show that $\delta$ is uniformly
bounded away from zero in order to finish the proof.
To this end we will establish the Morrey-type estimate
 \begin{equation}\label{eq:Campanato}
  \lnorm[\infty]{\gamma'(\cdot+w) - \gamma'(\cdot)} \leq C \E(\g)^{1/q} |w|^{\alpha}
  \qquad\text{for all }w\in[-\tfrac12,\tfrac12]
 \end{equation}
 where $\alpha = 3(p-1)/q-2>0$.
 As $\delta$ was chosen to be maximal with respect to~\eqref{eq:delta=max},
 we arrive at
 \begin{equation*}
  \tfrac1{10} \le C \E(\g)^{1/q} \delta^{\alpha}
 \end{equation*}
 which, applied to~\eqref{eq:semi-intM}, gives~\eqref{eq:energy-space}.

It remains to prove~\eqref{eq:Campanato} which follows
by standard arguments due to Campanato~\cite{campanato}.
Let $\g'_{B_r (x)}$ denote the integral mean 
 of $\g'$ over $B_r(x)$. We calculate for $x \in \mathbb R / \mathbb Z$
 and $r \in (0,\delta)$ 
 \begin{align*}
  \frac 1 {2r}\int_{B_r(x)} |\gamma'(v)- \g'_{B_r(x)}| \d v
  &\leq \frac 1 {4r^2} \int_{B_r(x)} \int_{B_r(x)}| \gamma'(v) - \gamma'(u)| \d u \d v
  \\
  &\leq  \left(\frac 1 {4r^2} \int_{B_r(x)} \int_{B_r(x)}| \gamma'(v) - \gamma'(u)|^q \d u \d v\right)^{1/q}
  \\
  &\leq C r^{\alpha} \left(\int_{B_r(x)} \int_{B_r(x)} \frac{| \gamma'(v) - \gamma'(u)|^q} {|u-v|^{3p-2q-1}}\d u \d v\right)^{1/q} 
  \\
  &\refeq[\le]{delta-est} C r^{\alpha} \E(\g)^{1/q}.
 \end{align*}
 As~\eqref{eq:Campanato} only involves the domain of $\g$
 up to a measure zero set, we may restrict to Lebesgue points.
 We choose two Lebesgue points $u,v \in \mathbb R / \Z$ of $\gamma'$
 with $r:=|u-v| \in (0,\tfrac\delta2)$. Then
 \begin{equation*}
  |\dg(u) - \dg(v)| 
  \leq \sum_{k=0}^\infty \left| \dg_{B_{2^{1-k}r}(u)} - \dg_{B_{2^{-k}r}(u)} \right|
  + \left| \dg_{B_{2r}(u)} - \dg_{B_{2r}(v)} \right|
  + \sum_{k=0}^\infty \left| \dg_{B_{2^{1-k}r}(v)} - \dg_{B_{2^{-k}r}(v)} \right|.
  \end{equation*}
  Since
  \begin{align*}
   \left| \dg_{B_{2r}(u)} - \dg_{B_{2r}(v)} \right| 
   &\leq \frac {\int_{B_{2r}(u)} |\dg(x) - \dg_{B_{2r}(u)}| \d x +
    \int_{B_{2r}(v)} |\dg(x) - \dg_{B_{2r}(v)}| \d x} {|B_{2r} (u) \cap B_{2r}(v)|}
    \\
    &\leq C |u-v|^\alpha \E(\g)^{1/q} 
  \end{align*}
  as $r=|u-v|$ and, for all $y \in \mathbb R / \mathbb Z$, $R\in(0,\tfrac\delta2)$,
  \begin{align*}
   \left| \dg_{B_{2R}(y)} - \dg_{B_{R}(y)} \right| 
   &\leq \frac {\int_{B_{R}(y)} |\dg(x) - \dg_{B_{2R}(y)}| \d x +
    \int_{B_{R}(y)} |\dg(x) - \dg_{B_{R}(y)}| \d x} {2R}
    \\
    &\leq C R^\alpha \E(\g)^{1/q},
  \end{align*}
  we deduce
   $|\dg(u) - \dg(v)| 
   \leq C \left( \sum_{k=0} ^\infty 2^{-k\alpha} + 1 +  \sum_{k=0} ^\infty 2^{-k\alpha}   \right)
   |u-v|^\alpha \E(\g)^{1/q}$.
  Thus
   $|\dg(u) - \dg(v)| \leq C |u-v|^\alpha \E(\g)^{1/q}$
 for all Lebesgue points of $\gamma'$ with $|u-v| < \frac\delta2$.
 The case $|u-v| \geq \frac\delta2$ follows by the triangle inequality.
\end{proof}

Let us conclude this section by briefly commenting on the other
ranges in the $(p,q)$-domain, see Figure~\ref{fig:range}.

\begin{remark}[Non-repulsive energies for $p<\frac23q+1$]\label{rem:non-rep}
 A bi-Lipschitz estimate is not guaranteed for injective curves if $p<\frac23q+1$. We briefly give the following example.
 Consider the curves $u\mapsto (u,0,0)$ and $u\mapsto (0,u,\delta)$ for $u\in[-1,1]$, $\delta\in[0,1]$.
 The interaction of these strands leads to the $\E$-value
 \begin{align*}
  &C\iiint_{[-1,1]^3} \frac{\br{\delta^{2}+u^{2}}^{q/2}}{\abs{v-w}^{p-q}{{\br{\delta^{2}+u^2+v^{2}}^{p/2}\br{\delta^{2}+u^2+w^{2}}^{p/2}}}}\d w\d v\d u \\
  &\le C\iiint_{[-1,1]^3} \frac{\br{\delta^{2}+u^{2}}^{(q-p)/2}}{\abs{v-w}^{p-q}{{\br{\delta^{2}+u^2+v^{2}+w^{2}}^{p/2}}}}\d w\d v\d u.
 \end{align*}
 Introducing polar coordinates $u=r\cos\th$, $v=r\sin\th\cos\varphi$,
 $w=r\sin\th\sin\varphi$, the former quantity is bounded by
 \begin{align*}
  & C\int_0^{\sqrt3}\int_0^{\pi}
  \frac{\br{\delta^{2}+r^{2}\cos^{2}\th}^{(q-p)/2}r^2\sin\th}
  {r^{p-q}\sin^{p-q}\th\br{\delta^{2}+r^{2}}^{p/2}}
  \d\th\d r
  \underbrace{\int_{0}^{2\pi}\frac{\d\varphi}{\abs{\cos\varphi-\sin\varphi}^{p-q}}}_{{\le C}} \\
  & \le C\int\limits_0^{\sqrt3}\br{\int\limits_{[0,\frac\pi4]\cup[\frac{3\pi}4,\pi]}
  \frac{\br{\delta^{2}+r^{2}}^{(q-p)/2}\d\th}
  {r^{p-q-2}\underbrace{\sin^{p-q-1}\th}_{\ge1}\br{\delta^{2}+r^{2}}^{p/2}}
  +\int\limits_{\frac\pi4}^{\frac{3\pi}4}
  \frac{\br{\delta^{2}+r^{2}\cos^{2}\th}^{(q-p)/2}r\sin\th}
  {r^{p-q-1}\sin^{p-q}\th\br{\delta^{2}+r^{2}}^{p/2}}
  \d\th
  }\d r \\
  & \le C\int_0^{\sqrt3}\br{
  \br{\delta+r}^{-3p+2q+2}
  +r^{-p+q+1}\br{\delta+r}^{-p}
  \int_{0}^{r}\br{\delta^{2}+\sigma^{2}}^{(q-p)/2}\d\sigma
   }\d r \\
  &\le C\br{1-\delta^{-3p+2q+3}}\le C.
 \end{align*}
 Using Theorem~\ref{thm:energy-space}
 and the monotonicity of $\E[\cdot,q]$ for fixed $q$, it is easy to
 produce a family of knots uniformly converging to a non-embedded curve
 without an energy blow-up as $\delta\searrow0$,
 so these energies are not self-repulsive.
\end{remark}

\begin{remark}[Singular energies for $p\ge q+\tfrac23$, $q>1$]\label{rem:singular}
 For $p\ge q+\tfrac23$, $q>1$, and an absolutely continuous $\g:\R/\Z\to\R$ we have
 $\E(\g)\equiv\infty$ for all $C^{1}$-curves~$\g$.
 To see this, note that we assumed $p<\frac23q+1$ in Theorem~\ref{thm:energy-space} mainly
because neither~\eqref{eq:Wsemi} nor~\eqref{eq:Wsemi'} is not defined for $s\ge1$.
For general $p\ge \frac23q+1$ we nevertheless still have
\begin{equation*}
 \int_{\R/\Z}\int_{-1/2}^{1/2} \frac{\abs{\dg(u+w)-\dg(u)}^q}{\abs{w}^{3p-2q-1}}\d w\d u
 \le C\br{\E(\g)+\E(\g)^{\beta}}.
\end{equation*}
 Applying Brezis~\cite[Prop.~2]{brezis}, the function $\dg$ is constant,
 hence $\g$ lies on a straight line. Therefore, $\g$ cannot be a closed
 $C^{1}$-curve.
\end{remark}

\begin{remark}[Strange energies for $p\in[q+\tfrac23,\tfrac23q+1)$]\label{rem:strange}
 On $p\in[q+\tfrac23,\tfrac23q+1)$, $p,q>0$, see the hatched area in Figure~\ref{fig:range},
 we find the strange behavior that there are no
 closed finite-energy $C^{3}$-curves
 while self-intersections, and in particular corners, are not penalized.
 So piecewise linear curves (polygonals) have finite energy.
 
 The latter can be seen by adapting the calculation
 in Remark~\ref{rem:non-rep}.
 For the former we recall that a closed
 arc-length parametrized $C^2$-curve must have
 positive curvature $\abs{\g''}$ at some point $u_{0}$ and by continuity
 there are $c,\delta>0$ with
 $\abs{\g''}\ge c>0$ on $[u_{0}-\delta,u_{0}+\delta]$.
 As $\g''\perp\dg$ we obtain $\abs{\g''\wedge\dg}=\abs{\g''}\ge c$.
 So $\E(\g)$ is bounded below by
 \begin{align*}
  &\int\limits_{u_{0}-\delta}^{u_{0}+\delta}
  \int\limits_{-\delta}^{\delta}
  \int\limits_{\frac13\abs v}^{\frac23\abs v}
  \abs v^{-3p+2q}\abs{\frac{\T v0}v\wedge\frac{\T 0w}w}^{q} \d w\d v\d u \\
  &=\int\limits_{u_{0}-\delta}^{u_{0}+\delta}
  \int\limits_{-\delta}^{\delta}
  \int\limits_{\frac13\abs v}^{\frac23\abs v}
  \abs v^{-3p+2q}\Bigg|\tfrac{v-w}2\g''(u)\wedge\dg(u)
  +\tfrac{v^{2}}2\int_{0}^{1}(1-\vth_{1})^{2}\g'''(u+\vth_{1}v)\d\vth_{1}\wedge\br{\dg(u)+\tfrac w2\g''(u)} \\
  &\qquad{}-\tfrac{w^{2}}2\int_{0}^{1}(1-\vth_{2})^{2}\g'''(u+\vth_{2}w)\d\vth_{1}\wedge\br{\dg(u)+\tfrac v2\g''(u)} \\
  &\qquad{}+\tfrac{v^{2}w^{2}}2\iint_{[0,1]^{2}}(1-\vth_{1})^{2}(1-\vth_{2})^{2}\g'''(u+\vth_{1}v)\wedge\g'''(u+\vth_{2}w)\d\vth_{1}\d\vth_{2}
  \Bigg|^{q} \d w\d v\d u \\
  &\ge \delta
  \bv{\tilde c-C\delta \lnorm[\infty]{\g'''}^{q}\br{\lnorm[\infty]{\g'''}+\lnorm[\infty]{\g''}+1}^{q}}
  \int_{-\delta}^{\delta}
  \abs v^{-3p+3q+1}\d v.
 \end{align*}
 Lessening $\delta>0$, the square bracket is positive.
 This gives $\E(\g)=\infty$.
\end{remark}

\section{\secB}

The arguments here are quite similar as
for the tangent-point energies~\cite{blatt-reiter3},
however, we provide full proofs for the readers' convenience.

Using Theorem~\ref{thm:energy-space} together with
the Arzel\`a-Ascoli theorem, we see that sets of 
curves in $\Cia(\mathbb R / \mathbb Z, \mathbb R^n)$
with a uniform bound on the energy are sequentially compact in $C^1$.
To this end we need the following result.

\begin{proposition}[Uniform bi-Lipschitz estimate]\label{prop:bi-Lipschitz}
 For every $M < \infty$ and~\eqref{eq:sub-critical} there is a constant $C(M,p,q)>0$ such 
 that every curve $\gamma\in \Cia(\mathbb R / \mathbb Z, \mathbb R^n)$
 parametrized by arc-length with
 \begin{equation}\label{eq:TP<M}
  \E(\gamma) \leq M
 \end{equation}
 satisfies the bi-Lipschitz estimate
 \begin{equation}\label{eq:bil}
  \abs{u-v} \leq C(M,p,q) \abs{\gamma(u) - \gamma(v)}
  \qquad \text{for all } u,v \in \mathbb R / \mathbb Z.
 \end{equation}
\end{proposition}

The proof is based on the following lemma.
To be able to state it,
we set for two arc-length parametrized curves $\gamma_i : I_i \rightarrow \mathbb R$, $i=1,2$,
$I_1,I_2$ open intervals,
\begin{align*}
 \E(\gamma_1, \gamma_2) &:= \E(\gamma_1) + \E(\gamma_2) + {} \\
 &{}\qquad +\iiint\limits_{I_1^2\times I_{2}}
 \frac{\abs{\dg_{1}(u_{1})}\abs{\dg_{1}(u_{2})}\abs{\dg_{2}(u_{3})}}{\RR(\g_{1}(u_{1}),\g_{1}(u_{2}),\g_{2}(u_{3}))}
 {\d u_{1}\d u_{2}\d u_{3}} \\
 &{}\qquad +\iiint\limits_{I_1\times I_{2}^{2}}
 \frac{\abs{\dg_{1}(u_{1})}\abs{\dg_{2}(u_{2})}\abs{\dg_{2}(u_{3})}}{\RR(\g_{1}(u_{1}),\g_{2}(u_{2}),\g_{2}(u_{3}))}
 {\d u_{1}\d u_{2}\d u_{3}}.
\end{align*}


\begin{lemma}\label{lem:last}
 Let $\alpha \in (0,1)$. For $\mu>0$ we let
 $M_\mu$ denote the set of all pairs $(\gamma_1, \gamma_2)$
 of curves $\gamma_i \in \Cia([-1,1],\R^n)$ 
 satisfying
 \begin{enumerate}
  \item $|\gamma_1(0) - \gamma_2 (0)| = 1$,
  \item $\g_{1}'(0)\perp\br{\gamma_1(0)-\gamma_2(0)} \perp \gamma_2'(0)$,
  \item $\|\gamma_{i}'\|_{C^{0,\alpha}} \leq \mu$, \qquad$i=1,2$.
 \end{enumerate}
 Then there is a $c=c(\a,\mu)>0$ such that
 \begin{equation*}
  \E(\gamma_1, \gamma_2) \geq c
  \qquad\text{for all }(\gamma_1, \gamma_2) \in M_\mu.
 \end{equation*}
\end{lemma}

\begin{figure}
 \centering
 \includegraphics[scale=.5]{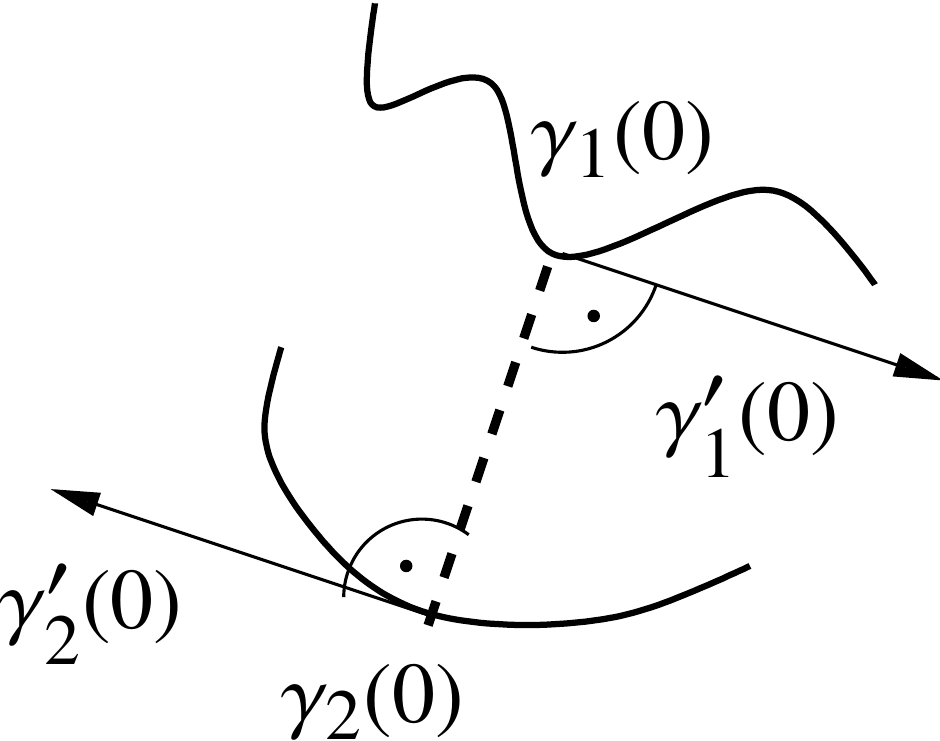}
 \caption{A pair of curves $(\g_1,\g_{2})\in M_{\mu}$
  defined in Lemma~\ref{lem:last}.
  Note that the arcs $\g_{1}$, $\g_{2}$ cannot intersect each other.}
\end{figure}

\begin{proof}
 It is easy to see that $\E(\gamma_1, \gamma_2)$
 is zero if and only if both $\gamma_1$ and $\gamma_2$ 
 are part of one single straight line.
 We will show that $\E(\cdot,\cdot)$ attains its minimum
 on $M_\mu$. As $M_\mu$ does not contain straight lines by~(i), (ii),
 this minimum is strictly positive which thus proves the lemma.
 
 Let $(\gamma_1^{(n)},\gamma_2^{(n)})$ be a minimizing sequence in 
 $M_\mu$, i.e.\@ we have
 \begin{equation*}
  \lim_{n\rightarrow \infty}\E (\gamma_1^{(n)},\gamma_2^{(n)})
  = \inf_{M_\mu} \E(\cdot,\cdot).
 \end{equation*}
 Subtracting $\g_{1}(0)$ from \emph{both} curves, i.e.\@ setting
 \begin{equation*}
  \tilde \gamma_i^{(n)}(\tau) := \gamma_i^{(n)}(\tau) - \gamma_{1}(0),
  \qquad i = 1,2,
 \end{equation*}
 and using Arzel\`a-Ascoli we can pass to a subsequence such that
 \begin{equation*}
  \tilde \gamma_i^{(n)} \to \tilde \gamma_i\qquad\text{in }C^{1}.
 \end{equation*}
 Furthermore, $(\tilde \gamma_1, \tilde \gamma_2) \in M_\mu$
 since $M_\mu$ is closed under convergence in $C^1$.
 Since, by Fatou's lemma, the functional $\E$ is lower semi-continuous
 with respect to $C^1$ convergence, we obtain
 \begin{equation*}
  \E(\tilde \gamma_1, \tilde \gamma_2) 
  \leq \lim_{n\rightarrow \infty} \E (\tilde \gamma_1^{(n)}, \tilde \gamma_2^{(n)})
  = \lim_{n\rightarrow \infty} \E ( \gamma_1^{(n)}, \gamma_2^{(n)})
  =  \inf_{M_\mu} \E(\cdot,\cdot).
  \end{equation*}
\end{proof}

Let us use this lemma to give the

\begin{proof}[Proposition~\ref{prop:bi-Lipschitz}]
Applying Theorem~\ref{thm:energy-space} to~\eqref{eq:TP<M} we obtain
 $\|\gamma'\|_{C^{0,\alpha}} \leq C(M)$
 for $\alpha = 3\frac {p-1} q -2\in(0,1-\frac1q)$.
As an immediate consequence there is a $\delta=\delta(\a,M)>0$
such that
\begin{equation}\label{eq:bil2}
 |u-v| \leq 2 |\gamma(u)- \gamma(v)|
\end{equation}
for all $u,v \in \mathbb R / \mathbb Z$ with $|u-v| \leq \delta.$
Let now
\[ S := \inf \sets{\!\rule{0ex}{1em}\abs{\gamma(u)-\gamma(v)}}{u,v \in \R/\Z,\abs{u-v}\geq\delta} \leq \tfrac12. \]
 We will complete the proof by estimating $S$ from below.
 Using the compactness of $\set{u,v \in \R/\Z,\abs{u-v}\geq\delta}$,
 there are $s,t \in \mathbb R / \mathbb Z$ with 
$|s-t| \geq \delta$ and
 $|\gamma(s)-\gamma(t)| =S$.
If now $|s-t| =\delta$ we obtain
\begin{equation*}
 2 S = 2|\gamma (s) - \gamma(t)| \refeq[\ge]{bil2} \delta
\end{equation*}
and hence
\begin{equation*}
 |u-v| \leq \tfrac12 \le \frac S\delta \le \frac{\abs{\g(u)-\g(v)}}{\delta(\a,M)}
\end{equation*}
for all $u,v \in \mathbb R / \mathbb Z$ with $|u-v| \geq \delta$.
This proves the proposition in this case.
If on the other hand $|s-t| > \delta$ then we infer using the
minimality of $|\gamma(s)-\gamma(t)|$
\begin{equation*}
 \gamma'(s) \perp (\gamma(s) - \gamma(t)) \perp \g'(t).
\end{equation*}
We define for $\tau \in [-1,1]$
\begin{equation*}
 \gamma_1 (\tau) := \frac 1 S \gamma(s+ S\tau)
 \qquad\text{and}\qquad
 \gamma_2(\tau) := \frac 1 S \gamma(t + S \tau).
\end{equation*}
Since 
$\| \gamma'_i\|_{C^{0,\alpha}} \leq \| \gamma'\|_{C^{0,\alpha}}
{\le C(M)}$
we may apply Lemma~\ref{lem:last} which yields
\begin{equation*}
 \E(\gamma_1,\gamma_2) \geq c(\a,M) >0.
\end{equation*}
Together with 
 $\E (\gamma_1, \gamma_2) \leq S^{3p-2q-3} \E (\gamma)$
this leads to
\begin{equation*}
 S \geq \left( \frac {c(\a,M)}{\E(\gamma)}\right)^{\frac 1 {3p-2q-3}}
 \ge\br{\frac{c(\a,M)}{M}}^{\frac1{3p-2q-3}}.
\end{equation*}
Hence,
$\abs{u-v} \le\tfrac12\le\frac{\abs{\g(u)-\g(v)}}{2S}
   \le C(M,p,q)\abs{\g(u)-\g(v)}$
   for all $u,v \in \mathbb R / \mathbb Z$ with $|u-v| \geq \delta$.
\end{proof}

We are now in the position to prove the
compactness result which is crucial both to
the existence of minimizers
in any knot class
and to the self-avoiding behavior of the energies.

\begin{proposition}[Sequential compactness]\label{prop:sequentiallycompact}
 For each $M< \infty$ the set
 \begin{equation*}
  A_M:=\sets{\gamma \in \Cia(\mathbb R / \mathbb Z, \mathbb R^n)}{ 
  \E(\gamma) \leq M}
 \end{equation*}
 is sequentially compact in $C^1$ up to translations.
\end{proposition}

\begin{proof}
 By Theorem~\ref{thm:energy-space}
 there are $C(M)< \infty$ and $\a=\a(p,q)>0$ such that
  $\|\dg\|_{C^\alpha} \leq C(M)$
 for all $\g\in A_M$ and hence
 \begin{equation*}
  \|\tilde \g\|_{C^{1,\alpha}} \leq C(M) + 1
 \end{equation*}
 where $\tilde \g (u) := \g(u)-\g(0)$.
 By Proposition~\ref{prop:bi-Lipschitz},
 the bi-Lipschitz estimate~\eqref{eq:bil} holds.
 
 Let now $\g_n \in A_M$. Then
 \begin{equation*}
  \|\tilde \g_n\|_{C^{1,\alpha}} \leq C(M) + 1
 \end{equation*}
 and hence after passing to suitable subsequence we have
  $\tilde \g_n \rightarrow \g_{0}$
 in $C^1$. Since $\g_n$ was parametrized by arc-length,
 $\g_0$ is still parametrized by arc-length and
 the bi-Lipschitz estimate carries over to $\g_{0}$.
 So, especially, $\g_{0} \in \Cia(\mathbb R / \mathbb Z, \mathbb R^n)$.
 From lower semi-continuity with respect 
 to $C^1$ convergence we infer
 \begin{equation*}
  \E(\g_0) \leq \liminf_{n\rightarrow \infty} \E(\g_n) \leq M.
 \end{equation*}
 So $\g_0 \in A_M$.
\end{proof}

We may now pass to the

\begin{proof}[Theorem~\ref{thm:existence}]
 Let $\seqn[k]{\g}\in \Cia$ be a minimal sequence for $\E$ in a given knot class $K$,
 i.e.\@ let
 \begin{equation*}
  \lim_{k\rightarrow \infty} \E(\g_{k}) = \inf_{\Cia\cap K} \E.
 \end{equation*}
 After passing to a subsequence and suitable translations, we hence get by Proposition~\ref{prop:sequentiallycompact} 
 a ~$\g_0\in \Cia$ with $\g_k\to\g_0$ in $C^1$.
 As the intersection of every knot class with $C^1$
 is an open set in $C^1$~\cite[Cor.~1.5]{blatt:isot}
 (see~\cite{reiter:isot} for an explicit construction),
 the curve $\g_0$ belongs to the same knot class
 as the elements of the minimal sequence $\seqn[k]{\g}$. 
 The lower semi-continuity of $\E$ furthermore implies
 \begin{equation*}
  \inf_{\Cia\cap K} \E\leq \E(\g_0)
  \leq  \lim_{n\rightarrow \infty} \E(\g_n)=\inf_{\Cia\cap K}\E.
 \end{equation*}
 Hence, $\g_0$ is the minimizer we have been searching for.
\end{proof}

By the same reasoning one derives the existence of a global minimizer of~$\E$.

Let us conclude this section by deriving that the generalized integral Menger curvature
are in fact knot energies (in the sub-critical range).

\begin{proposition}[$\E$ is a strong knot energy {\cite[Cor.~2.3]{SM9}}]\label{prop:strong-knot-energy}
Let~\eqref{eq:sub-critical} hold.\\[-2em]
 \begin{enumerate}
  \item If $\seqn[k]{\g}$ 
  is a sequence of embedded $\W$-curves uniformly converging to a non-injective curve $\g_\infty\in C^{0,1}$ parametrized by arc-length then $\E(\g_k)\to\infty$.
  \item For given $E,L>0$ there are only finitely many knot types having a representative with $\E<E$ and $\text{length}=L$.
 \end{enumerate}
\end{proposition}

\begin{proof}
 The first statement immediately follows from the bi-Lipschitz estimate
 in Proposition~\ref{prop:bi-Lipschitz}, as a sequence with bounded energy 
 would be sequentially compact in $\Cia$ and thus cannot uniformly converge 
 to a non-injective curve.   
 
 To show the second statement, let us assume that it was wrong, i.e.,
 that there are curves $\seqn[k]{\g}$ of length $L$,
 all belonging to different knot classes,
 with energy less than $E$. Of course we can assume that
 $L=1$.
 After applying suitable transformations and passing to a subsequence,
 Proposition~\ref{prop:sequentiallycompact} guarantees the existence
 of $\g_{0}\in A_{M}$ with $\g_{k}\to\g_{0}$ in $C^1$.
 Again by~\cite{blatt:isot,reiter:isot}
 this implies that almost all $\g_k$ belong to the
 same knot class as $\g_0$, which is a contradiction.
\end{proof}

%% file: regularity.tex
\section{\secC}

Recall that we have for $v,w \in D$
\begin{equation*}
 |v|,|w| \leq |v-w| \leq \tfrac 2 3.
\end{equation*}
Hence,  we obtain for a each curve $\g \in \Cia[0,1](\mathbb R / \mathbb Z)$
with $\E(\g) < \infty $ due to the bi-Lipschitz estimate
\begin{equation*}
 |\g(u+v) - \g(u)| \simeq |v|, \quad \quad
 |\g(u+w) - \g(u)| \simeq |w|, \quad \quad
 |\g(u+v) - \g(u+w)| \simeq |v-w|
\end{equation*}
for all $(u,v,w) \in \mathbb R / \mathbb Z \times D$.
(Here $a\simeq b$ is an abbreviation for the existence of uniform
constants $0<c\le C<\infty$ with $cb\le a\le Cb$.)

We derive the following form for the first variation which is at first site 
much more complicated than the formula derived by Hermes~\cite{hermes}, but due to the special structure of $D$
it is easier to do estimates using this formula.
We abbreviate
\[ R^{p,q}(u_{1},u_{2},u_{3}) := \RR(\g(u_{1}),\g(u_{2}),\g(u_{3})) \]
and we still use
\begin{equation}\tag{\ref{eq:short-notation}}
 \T[\bullet]{v}{w}:=\D[\bullet]vw.
\end{equation}
In contrast to O'Hara's knot energies, we can use a rather direct argument to deduce that
the integral Menger curvature is G\^ateaux differentiable by investigating the integrand, i.e.\@ by looking at the 
Lagrangian
\begin{equation*}
 L(\gamma)(u,v,w) := \frac {\left|\T w0 \wedge \T v0 \right|^q}{|\T w0|^{q}|\T v0|^q |\T vw|^q}|\gamma'(u)| |\gamma'(u+v)| |\gamma'(u+w)|.
\end{equation*}
For $\gamma, h \in \W$ as in the statement of the lemma and $\gamma_\tau := \gamma + \tau h$ one calculates
\begin{align*}
 &\delta L(\gamma;h) (u,v,w) :=\left.  \frac {\partial}{\partial \tau }\left(L_{\gamma_\tau}(u,v,w) \right) \right|_{\tau =0} \\
  &= \Bigg\{ q \left| \T w0 \wedge \T v0 \right|^{q-2}  \frac {\sp{\T w0 \wedge \T v0 ,\T[h] w0 \wedge \T v0 + \T w0 \wedge \T[h] v0 }}{|\T w0|^{p}|\T v0|^p |\T vw|^p} \\
  &- p \frac {\left|\T w0 \wedge \T v0 \right|^q}{|\T w0|^{p+2}|\T v0|^p |\T vw|^p}  \cdot  \sp{\T w0, \T[h] w0} \\
  &- p \frac {\left|\T w0 \wedge \T v0 \right|^q}{|\T w0|^{p}|\T v0|^{p+2} |\T vw|^p}  \cdot  \sp{\T v0, \T[h] v0} \\
  &- p \frac {\left|\T w0 \wedge \T v0 \right|^q}{|\T w0|^{p}|\T v0|^p |\T vw|^{p+2}}  \cdot  \sp{\T vw, \T[h] vw} \\
   &+ R^{p,q}(u, u+w, u+v) \sp{\frac {\dg (u)}{\abs{\dg (u)} }, \frac{\dh(u)}{\abs{\dg (u)} }} \\
  &+ R^{p,q}(u, u+w, u+v) \sp{\frac {\dg (u+v)}{\abs{\dg (u+v)} }, \frac{\dh(u+v)}{\abs{\dg (u+v)} }} \\
  &+ R^{p,q}(u, u+w, u+v) \sp{\frac {\dg (u+w)}{\abs{\dg (u+w)} }, \frac{\dh(u+w)}{\abs{\dg (u+w)} }} \Bigg\} |\dg(u+w)| |\dg(u+v)| |\dg(u)|.
\end{align*}
For future reference, we denote the seven terms one obtains from this formula after factoring out the 
outermost bracket by $\delta L_1, \ldots, \delta L_7$.

\begin{lemma}\label{lem:diff}
 Let \eqref{eq:sub-critical} hold and $\gamma \in \W\rzd$ be an injective regular curve. Then $\E$
 is G\^ateaux differentiable in $\gamma$ and 
 the first variation in direction $h \in \W \rzd$ is given by
 \begin{align}\label{eq:diff}
  \delta \E(\gamma;h) = 6 \iiint_{\mathbb R / \mathbb Z \times D}\delta L(\gamma;h)(u,v,w) \d w \d v \d u
\end{align}
and~\eqref{eq:dE-intro} holds.
\end{lemma}

\begin{proof}
Let $U$ be a neighborhood of $\gamma$ in $\W \subset C^{(3p-3)/q-1} \subset C^1$ consisting only of regular curves with
\begin{equation*}
 \inf_{\tilde \gamma \in U, u \in \mathbb R / \mathbb Z} |\tilde \gamma '(u)| =: M_1 >0 
\end{equation*}
and
\begin{equation*}
 \sup_{\tilde \gamma \in U, u\not=v \in \mathbb R / \mathbb Z} \frac {|\tilde \gamma(u) - \tilde \gamma(v)|}{|u-v|} =: M_2 < \infty. 
\end{equation*}
Using 
\begin{multline*}
 \sp{\frac {\T w0} w \wedge \frac{\T v0} v ,\frac {\T[h] w0} w \wedge \frac {\T v0} v 
 + \frac{\T w0} w\wedge \frac {\T[h] v0} v }
 \\= \sp{\left( \frac{\T w0} w - \frac {\T v0}v \right) \wedge \frac{\T v0} v ,
	\left( \frac{\T[h] w0} w - \frac {\T[h] v0}v \right) \wedge \frac{\T v0} v  
	+ \left( \frac{\T w0} w - \frac {\T v0}v \right) \wedge \frac{\T[h] v0} v},
\end{multline*}
\begin{equation*}
 \frac {\T 0w} w \wedge \frac { \T  0v} v = \left( \frac {\T 0w} w- \frac{\T 0v}v \right) \wedge \frac{\T 0v}v,
\end{equation*}
and
\begin{equation*}
 R^{p,q}(u,v,w) = \frac {\left|\left( \frac {\T 0w} w- \frac{\T 0v}v \right) \wedge \T 0v\right|^q}{|w|^{-q} \abs{\T0w}^{p}\abs{\T v0}^{p}\abs{\T vw}^{p}}
\end{equation*}
together with the bi-Lipschitz estimate we infer for $\tilde \gamma \in U$
\begin{equation}  \label{eq:estimateintegrand}
  \left|\delta L(\tilde \gamma; h)(u,v,w) \right|
 \leq C  \frac{|\frac {\T[\tilde \gamma] w0} w - \frac {\T[\tilde \gamma] v0} v|^q \|h'\|_{L^\infty} +
  |\frac {\T[\tilde \gamma] w0} w - \frac {\T[\tilde \gamma] v0} v|^{q-1}|\frac {\T[h] w0} w - \frac {\T[h] v0} v|}{|w|^{p-q} |v|^{p-q} |v-w|^p}.
\end{equation}
So for $0<|\tau|\leq 1$ so small that $\gamma_\tau \in U$
\begin{equation} \label{eq:majorant}
\begin{split}
 &\left|\frac {\partial}{\partial \tau }L(\gamma_\tau)(u,v,w) \right| \\
 &\leq C  \frac{|\frac {\T[\gamma] w0} w - \frac {\T[\gamma] v0} v|^q \|h'\|_{L^\infty} + |\frac {\T[h] w0} w - \frac {\T[h] v0} v|^q \|h'\|_{L^\infty} +
  |\frac {\T[\gamma] w0} w - \frac {\T[\gamma] v0} v|^{q-1}|\frac {\T[h] w0} w - \frac {\T[h] v0} v| + |\frac {\T[h] w0} w - \frac {\T[h] v0} v| ^q }{|w|^{p-q} |v|^{p-q} |v-w|^p}
 \\ &=: g(u,v,w)
 \end{split}
\end{equation}
where $g$ does not depend on $\tau$ and $C$ depends on $M_1, M_2, p, q$ only.

For $f \in \W \rzd$, we have
\begin{equation}\label{eq:esttripleintegral}
\begin{split}
   &\iiint_{\mathbb R / \mathbb Z \times D }\frac{|\frac {\T[f] w0} w - \frac {\T[f] v0} v|^q  }
   {|w|^{p-q} |v|^{p-q} |v-w|^p} \d v \d w \d u \\
  & = \iiint\limits_{\mathbb R / \mathbb Z \times D } 
   \frac{\abs{\int_{0}^{1}\br{f'(u+\theta w)-f'(u+\theta v)}\d\theta}^{q}}{ |w|^{p-q} |v|^{p-q} |w-v|^p} \d w \d v \d u \\
   & \le  C \int_0^{1}\iiint_{\R/\Z\times D } \frac{\abs{f'(u+\theta(w-v))-f'(u)}^{q}}{ |w|^{p-q} |v|^{p-q} |w-v|^p} \d w \d v \d u \d\theta \\
   &\refeq[\le]{intM<seminorm} C\norm f_{{\W}}^{q}.
   \end{split}
\end{equation}
Hence,
\begin{equation}
\begin{split} \label{eq:L1majorant}
&\iiint\limits_{\mathbb R / \mathbb Z \times D }\abs{g(u,v,w)}\d w\d v\d u \\
   &\le C\seminorm \g_{\W[(3p-q-2)/q,q]}^q \|h'\|_{L^\infty} + {}
  \seminorm {h}_{\W[(3p-q-2)/q,q]}^q \br{1+\lnorm[\infty]{h'}}
   \\ & \quad +   C  \left( \iiint\limits_{\mathbb R / \mathbb Z \times D} \frac{|\frac {\T w0} w - \frac {\T v0} v|^q}{|w|^{p-q} |v|^{p-q} |v-w|^p}\d w\d v\d u \right)^{1-\frac 1 q}
   \left( \iiint\limits_{\mathbb R / \mathbb Z \times D} \frac{|\frac {\T[h] w0} w - \frac {\T[h] v0} v|^q}{|w|^{p-q} |v|^{p-q} |v-w|^p} \d w\d v\d u\right)^{\frac 1 q}
  \\
 & \le C \br{\seminorm{\g}_{\W[(3p-2)/q-1,q]}^{q} + \seminorm{h}_{\W[(3p-2)/q-1,q]}^{q}}\br{1+\lnorm[\infty]{h'}}
 +C \seminorm \g_{\W[(3p-q-2)/q,q]}^{q-1}\seminorm {h}_{\W[(3p-q-2)/q,q]}^q.
\end{split}
\end{equation}
So $\delta L(\gamma_\tau; h)$ has a uniform $L^1$-majorant for $\tau$ sufficiently small.
Therefore,
by Lebegue's theorem of dominanted convergence,
we finally can use the fundamental theorem of calculus to write for $\tau$ small
\begin{align*}
 \frac {\E(\gamma+\tau h) - \E(\gamma)}{\tau} 
 &= 6 \iiint\limits_{\mathbb R / \mathbb Z \times D} \int_0^1 \delta L(\gamma_{s \tau};h)(u,v,w) \d s\d u \d v \d w \\
 & \xrightarrow{\tau \rightarrow 0} 6 \iiint\limits_{\mathbb R / \mathbb Z \times D} \delta L (\gamma; h)(u,v,w) \d u \d v \d w.
\end{align*}
Consequently, the first variation exists and has the form~\eqref{eq:diff}.

Using once more the symmetry of the integrand, we can bring this into 
the form~\eqref{eq:dE-intro} as follows.
%
%
Due to the symmetry of the integrand we have
$L\circ P_{\s}=L$ for any permutation matrix $P_{\s}$, $\in\mathfrak S_{3}$.
So we obtain
\begin{equation*}
 6\iiint_{\mathbb R/ Z \times D} \delta L(\gamma; h) 
 =  \sum_{\sigma \in\mathfrak S_{3} } \iiint_{P_{\s}(\mathbb R/ Z \times D)} \delta L(\gamma;h)
 = \iiint_{(\mathbb R / \mathbb Z) ^3} \delta L (\gamma;h).
\end{equation*}
The symmetry of $\RR$ now leads to the desired.

Furthermore, by~\eqref{eq:estimateintegrand},
the first variation defines a bounded linear operator on $\W$.
Hence $\E$ is G\^ateaux differentiable.
\end{proof}

In fact, we can even show that $\E$ is $C^1$,
though we will not use this fact in the rest of this article.

\begin{lemma}
 The functional $\E$ is $C^1$ on the subspace of all regular embedded
 $\W$-curves.
\end{lemma}

\begin{proof}
We will prove by contradiction that $\delta \E$
is a continuous map from embedded regular $\W$-curves
to $\br\W^\ast$.
So let us assume that $\delta \E$
was not continuous in $\gamma_{0}$.
Consequently, there are some $\eps_{0}>0$ and sequences $\seqn\g,\seqn h\subset\W$,
$\g_k\to\g_{0}$ in $\W$, $\norm{h_{k}}_{\W}\le1$, with
\begin{equation}\label{eq:contrassumpt}
 |\delta \E(\g_k; h_k) - \E (\gamma;h_k)| \geq \varepsilon_0.
\end{equation}

As in the proof of Lemma~\ref{lem:diff}, we can exploit the embedding $\W \hookrightarrow C^1$ and the openness of the set of regular embedded curves in $C^1$,
to find an open neighborhood $U$ of $\g_{0}$ consisting only of embedded curves, such that
\begin{equation*}
 \inf_{\tilde \gamma \in U, u \in \mathbb R / \mathbb Z} |\tilde \gamma '(u)| =: M_1 >0 
\end{equation*}
and
\begin{equation*}
 \sup_{\tilde \gamma \in U} \|\tilde \gamma\|_{\W} +\sup_{\tilde \gamma \in U, u\not=v \in \mathbb R / \mathbb Z} \frac {|\tilde \gamma(u) - \tilde \gamma(v)|}{|u-v|} =: M_2 < \infty. 
\end{equation*}
After passing to a subsequence
we may assume $\seqn\g\subset U$ and $h_{k}\to h_{0}\in\W$ in $C^{1}$
due to the compactness of the embedding $\W \hookrightarrow C^1$
which then also gives
\begin{equation*}
 \delta L(\g_k;h_k) - \delta L(\g_0;h_0) \rightarrow 0
\end{equation*}
pointwise almost everywhere on $\R/\Z\times D$ and hence in measure, i.e., for all $\varepsilon >0$ we have
\begin{equation}\label{eq:lebto0}
 \lim_{k \rightarrow \infty} \mathcal L^3 \br{A_{\eps,k}}=0,
\end{equation}
where $\mathcal L^3$ denotes the Lebesgue measure
and
\begin{equation*}\delimitershortfall-1pt
 A_{\eps,k}:= \sett{(u,v,w) \in \mathbb R / \mathbb Z \times D}
 {\abs{\delta L(\g_k;h_k)(u,v,w)-\delta L(\g_0;h_0)(u,v,w)} \geq \varepsilon}. 
\end{equation*}

For all $\varepsilon >0$ we can deduce from \eqref{eq:estimateintegrand}
using Young's inequality that there is a $C_{\eps}>0$ with
\begin{equation}
\begin{split}
 \left|\delta L(\gamma_{k};h_{k})(u,v,w) \right|
 &\leq   C_\varepsilon  \frac{|\frac {\T[\g_k] w0} w - \frac {\T[\g_k] v0} v|^q }{|w|^{p-q} |v|^{p-q} |v-w|^p} 
  + \varepsilon \frac { |\frac {\T[h_{k}] w0} w - \frac {\T[h_{k}] v0} v|^q}{|w|^{p-q} |v|^{p-q} |v-w|^p} 
  \\ &=: C_{\eps}g^{(1)}_k(u,v,w) + \varepsilon g^{(2)}_k(u,v,w)
 \end{split} 
\end{equation}
for all $k\in\N\cup\set0$.
Since the first summand converges in $L^1$ as $k\to\infty$, it is uniformly integrable,
so 
there is a $\delta_\eps >0$ such that
$\mathcal L^{3} (E) \leq \delta_\eps$ for any measurable subset $E \subset \mathbb R / \mathbb Z \times D$ implies
\begin{equation}
 \iiint_E g_k^{(1)}  \leq \frac\eps{C_{\eps}} \qquad\text{for all }k\in\N\cup\set0.
\end{equation}
Furthermore we infer from \eqref{eq:intM<seminorm}
\begin{equation*}
 \iiint_{\mathbb R / \mathbb Z \times D} g^{(2)}_k \d u\d v\d w \leq C \norm {h_{k}}_{\W} \leq C.
\end{equation*}
By~\eqref{eq:lebto0} there is some $k_0=k_{0}(\eps) \in \mathbb N$ with
\begin{equation*}
 \mathcal L^3 \br{A_{\eps,k}} \leq \delta_\eps
 \qquad\text{for all } k \geq k_0
\end{equation*}
which yields for
$k \geq k_0$ and
$B_{\eps,k} := \mathbb R / \mathbb Z \times D \setminus A_{\eps,k}$
\begin{align*}
 &\tfrac16|\delta \E(\g_k; h_k) - \delta\E (\gamma_{0};h_{0})| \\
 &\leq \iiint_{A_{\eps,k}} |\delta L(\g_k;h_k)-\delta L(\g_0;h_0) |
  + \iiint_{B_{\eps,k}} |\delta L(\g_k;h_k)-\delta L(\g_0;h_0) | \\
 & \leq \iiint_{A_{\eps,k}} \left(C_{\eps} g^{(1)}_k +  \varepsilon g^{(2)}_k \right)
 + \iiint_{A_{\eps,k}} \left(C_{\eps} g_{0}^{(1)} +  \varepsilon g_{0}^{(2)} \right)
 + \mathcal L^3(B_{\eps,k}) \varepsilon \\
 & 
  \leq C\eps.
\end{align*}
Hence,
\begin{align*}
 \eps_{0}&\refeq[\le]{contrassumpt}\abs{\delta \E(\g_k; h_k) - \E (\gamma;h_k)} \\
 &\le \abs{\delta \E(\g_k; h_k) - \E (\gamma;h)} + \abs{\delta \E(\g; h_k) - \E (\gamma;h)} \\
 &\le C\eps + C\norm{h_{k}-h}_{\W}
\end{align*}
for all $\varepsilon >0$ and $k\ge k_{0}(\eps)$ which leads to a contradiction.
\end{proof}

\section{\secD}

For the rest of this section, let us restrict to the case that $ \g$ is parametrized by arc-length.
Then we get using Lemma \ref{lem:diff} 
\begin{align*}
 \delta \E (\gamma;h) := 6q \tilde Q^{p,q} (\gamma,h) + 6 R^{p,q}_1 (\gamma h)
\end{align*}
where
\begin{multline*}
 \tilde Q^{p,q}(\gamma,h):= \iiint_{\mathbb R / \mathbb Z \times D} \left|\T w0 \wedge \T v0\right|^{q-2}  \frac {\sp{\T w0 \wedge \T v0 ,\T[h] w0 \wedge \T v0 + \T w0 \wedge \T[h] v0 }}{|\T w0|^{p}|\T v0|^p |\T vw|^p}   
\d w \d v \d u
\end{multline*}
and \begin{align*}
  R_1 (\gamma,h) := \iiint_{\mathbb R / \mathbb Z \times D} \Bigg(
  &- p \frac {\left|\T w0 \wedge \T v0 \right|^q}{|\T w0|^{p+2}|\T v0|^p |\T vw|^p}  \cdot  \sp{\T w0, \T[h] w0} \\
  &- p \frac {\left|\T w0 \wedge \T v0 \right|^q}{|\T w0|^{p}|\T v0|^{p+2} |\T vw|^p}  \cdot  \sp{\T v0, \T[h] v0} \\
  &- p \frac {\left|\T w0 \wedge \T v0 \right|^q}{|\T w0|^{p}|\T v0|^p |\T vw|^{p+2}}  \cdot  \sp{\T vw, \T[h] vw} \\
  &+ R^{p,q}(u, u+w, u+v) \sp{\dg (u), \dh(u)} \\
  &+ R^{p,q}(u, u+w, u+v) \sp{\dg (u+v),\dh(u+v)} \\
  &+ R^{p,q}(u, u+w, u+v) \sp{\dg (u+w),\dh(u+w)}
 \Bigg)\d w \d v \d u.
\end{align*}

For $q=2$ will see that $\tilde Q^p := \tilde Q^{p,2}$ contains the highest order term of the Euler-Lagrange operator.
To see this we use
$$
 \left\langle a\wedge b , c\wedge d  \right\rangle = \det \begin{pmatrix}\sp{a,c} & \sp{a,d} \\
                                                           \sp{b,c} & \sp{b,d}
                                                          \end{pmatrix} 
$$
to get 
\begin{equation}\label{eq:abac}
 \left\langle a \wedge b, a \wedge c\right\rangle = \sp{a,a}\sp{c,b} - \sp{a,c}\sp{a,b} 
 = |a|^2\left\langle P^\bot_{a} b, c \right\rangle.
\end{equation}
Hence,
\begin{align*}
\tilde Q^{p}(\gamma;h) &= \iiint\limits_{\mathbb R / \mathbb Z \times D } \Bigg(
 \frac{\sp{P^\bot_{\T v0}\T w0,\T[h] w0}}{\abs{\T vw}^p\abs{\T v0}^{p-2}\abs{\T w0}^p}
+
 \frac{\sp{P^\bot_{\T w0}\T v0,\T[h] v0}}{\abs{\T wv}^p\abs{\T v0}^p\abs{\T w0}^{p-2}}
 \Bigg)\d w \d v \d u
\\
&=    \iiint\limits_{\mathbb R / \mathbb Z \times D } \Bigg( \frac{\sp{P^\bot_{\T v0}\T w0,\T[h] w0}}{\abs{v-w}^p\abs{v}^{p-2}\abs{w}^p}
+
 \frac{\sp{P^\bot_{\T w0}\T v0,\T[h] v0}}{\abs{v-w}^p\abs{v}^p\abs{w}^{p-2}}
 \Bigg)\d w \d v \d u + R_2(\gamma;h) \\
& =   \iiint\limits_{\mathbb R / \mathbb Z \times D } \Bigg( \frac{\sp{P^\bot_{\T v0} \br{\frac {\T w0}{w}
- \frac {\T v0}{v}},\frac{\T[h] w0}{w}}}{\abs{v-w}^p\abs{v}^{p-2}\abs{w}^{p-2}}
-
 \frac{\sp{P^\bot_{\T w0} \br{\frac {\T w0}{w}
- \frac {\T v0}{v}},\frac{\T[h] v0}{v}}}{\abs{v-w}^p\abs{v}^{p-2}\abs{w}^{p-2}}
 \Bigg)\d w \d v \d u + R_2(\gamma;h) \\
& =   \iiint\limits_{\mathbb R / \mathbb Z \times D } \Bigg( \frac{\sp{\frac {\T w0}{w}
- \frac {\T v0}{v},\frac{\T[h] w0}{w}-\frac{\T[h] v0}{v}}}{\abs{v-w}^p\abs{v}^{p-2}\abs{w}^{p-2}}
 \Bigg)\d w \d v \d u + R_2(\gamma;h) - R_3(\gamma;h) \\
\end{align*}
where
\begin{align*}
 R_2 (\gamma;h) := \iiint\limits_{\mathbb R / \mathbb Z \times D } \Bigg(
 &\sp{P^\bot_{\T v0}\T w0,\T[h] w0}  \left( \frac 1{\abs{\T vw}^p\abs{\T v0}^{p-2}\abs{\T w0}^p} 
 - \frac 1{\abs{v-w}^p\abs{v}^{p-2}\abs{w}^p}\right)
 \\ &+
 \sp{P^\bot_{\T w0}\T v0,\T[h] v0} \left( \frac 1 {\abs{\T wv}^p\abs{\T v0}^p\abs{\T w0}^{p-2}}
  - \frac 1 {\abs{v-w}^p\abs{v}^p\abs{w}^{p-2}}\right)
 \Bigg)\d w \d v \d u
\end{align*}
and
\begin{align*}
 R_3 (\gamma;h) :=   \iiint\limits_{\mathbb R / \mathbb Z \times D } \Bigg( \frac{\sp{P^{\top}_{\T v0}
\left( \frac{ \T w0 } w - \frac{ \T v0 } v \right),\frac{\T[h] w0}w}}{\abs{v-w}^p\abs{v}^{p-2}\abs{w}^{p-2}}
+
 \frac{\sp{P^{\top}_{\T w0}\left(\frac {\T v0}{v} -\frac {\T w0}{w}\right) ,\frac{\T[h] v0}v}}{\abs{v-w}^p\abs{v}^{p-2}\abs{w}^{p-2}}
 \Bigg)\d w \d v \d u.
\end{align*}
Using
\begin{align*}
 Q^{(p)}(\gamma;h) :=   \iiint\limits_{\mathbb R / \mathbb Z \times D } \Bigg( \frac{\sp{\frac {\T w0}{w}
- \frac {\T v0}{v},\frac{\T[h] w0}{w}-\frac{\T[h] v0}{v}}}{\abs{v-w}^p\abs{v}^{p-2}\abs{w}^{p-2}}
 \Bigg)\d w \d v \d u
\end{align*}
we hence get
\begin{equation}\label{eq:Re}
 \delta M^{p,2} = 12 \left( Q^{(p)} + \tfrac12R_1 +R_2 + R_3  \right).
\end{equation}

\begin{proposition}\label{prop:Q}
 The functional $\Q$ is bilinear on $\br{\W[3p/2-2,2]}^{2}$, more precisely
 \[ \Q(f,g) = \sum_{k\in\Z} \rho_k\sp{\hat f_k,\hat g_k}_{\C^d} \qquad \text{where } \rho_k = c\abs k^{3p-4} + o\br{\abs k^{3p-4}} \text{ as }\abs k\nearrow\infty \]
 and $c>0$.
 Here $\hat\cdot_k$ denotes the $k$-th Fourier coefficient
 \[
  \hat f_k := \int_{0}^1 f(x) e^{-2\pi ikx } \d x.
 \]
\end{proposition}

\begin{proof}
   Testing with the basis ${\rm e}_l \cdot e^{2\pi i k x}$ of $L^2$, $l=1,2,3$, $k\in \mathbb Z$, where ${\rm e}_1, {\rm e}_2, {\rm e}_3$ is the standard basis of $\mathbb R ^3$,
   we get
\begin{equation*}
      \Q(f,g) =
       \sum_{k\in\Z} \sp{\hat f_k,\hat g_k}_{\C^d} \rho_k
 \end{equation*}
 where
 \begin{equation*} 
       \rho_k :=  
       \iint\limits_{D} 
       \frac{ \abs{ \frac {e^{2\pi i  kw} -1} {w} - \frac {e^{2\pi i k v} -1} {v} }^2}{|v|^{p-2}|w|^{p-2}|v-w|^{p}} \d w\d v.
  \end{equation*}
 A simple substitution leads to 
 \begin{align*}
  \rho_k = \abs k^{3p - 4}
  \iint\limits_{D_k} 
       \frac{ \abs{ \frac {e^{2\pi i  w} -1} {w} - \frac {e^{2\pi i  v} -1} {v}}^2}{|v|^{p-2}|w|^{p-2}|v-w|^{p}} \d w\d v 
 \end{align*}
 where $D_k := k \cdot D.$ We use the fundamental thorem of calculus and Jensen's inequality to get
 \begin{align*}
  \frac{ \abs{ \frac {e^{2\pi i  w} -1} {w} - \frac {e^{2\pi i  v} -1} {v}}^2}{|v|^{p-2}|w|^{p-2}|v-w|^{p}}
  &\leq 4\pi^{2}\int_0^1 \frac {|e^{2\pi i \theta w} - e^{2\pi i \theta v} |^2} {|w|^{p-2} |v|^{p-2} |v-w|^{p}} \d \theta\\
  &= 4\pi^{2} \int_0^1 \frac {|e^{2\pi i \theta (v-w)} - 1 |^2} {|w|^{p-2} |v|^{p-2} |v-w|^{p}} \d \theta.
 \end{align*}
 Hence, 
 \begin{align*}
  \iint\limits_{v<0, w>0} \frac{ \abs{ \frac {e^{2\pi i k w} -1} {w} - \frac {e^{2\pi i k v} -1} {v}}^2}{|v|^{p-2}|w|^{p-2}|v-w|^{p}}
  &\leq C\iint\limits_{v<0, w>0} \int_0^1 \frac {|e^{2\pi i \theta (v-w)} - 1 |^2} {|w|^{p-2} |v|^{p-2} |v-w|^{p}} \d \theta \d w \d v
  \\ &= C\left( \int_0^1  {\theta^{3p-6}} \d \theta \right) \iint\limits_{v<0, w>0}  \frac {|e^{2\pi i  (v-w)} - 1 |^2} {|w|^{p-2} |v|^{p-2} |v-w|^{p}}  \d w \d v
  \\ &\leq C \iint\limits_{v<0, w>0}  \frac {|e^{2\pi i  (v-w)} - 1 |^2} {|w|^{p-2} |v|^{p-2} |v-w|^{p}}  \d w \d v
  \\ & = C \int_0^\infty \int_{0}^{\tilde w} \frac {|e^{-2\pi i \tilde w} - 1 |^2} {|\tilde w -v|^{p-2} |v|^{p-2} |\tilde w|^{p}}  \d v \d \tilde w
  \\ & = C  \int_0^\infty \int_0^1  \frac {|e^{-2\pi i \tilde w} - 1 |^2} {|1 -t|^{p-2} |t|^{p-2} |\tilde w|^{3p-5}}  \d t \d \tilde w
  \\ & \le C \int_0^\infty \frac {1-\cos2\pi\tilde w} {|\tilde w|^{3p-5}} \d \tilde w < \infty.
 \end{align*}
 Thus, we have shown that
 \begin{equation*}
  \frac {\rho_k} {|k|^{3k-4}} \xrightarrow{|k| \rightarrow \infty} \iint_{v<0, w>0} \frac{ \abs{ \frac {e^{2\pi i k w} -1} {w} - \frac {e^{2\pi i k v} -1} {v}}^2}{|v|^{p-2}|w|^{p-2}|v-w|^{p}}\d w\d v \in (0, \infty).
 \end{equation*}
\end{proof}

In the following statement, we use the symbol $\otimes$ to
denote any kind of product structure, such as cross product, dot product,
scalar or matrix multiplication.

\begin{lemma}
 The term $R^{(p)} := \tfrac12R_1 + R_2 + R_3$ is a finite sum of terms of the form
 $$
  \iiint\limits_{\mathbb R / \mathbb Z \times D} \idotsint\limits_{[0,1]^K} g^p(u,v,w;s_{1},\dots,s_{K-2}) \otimes h'(u+s_{K-1}v + s_K w) 
  \d\theta_{1}\cdots\d\theta_{K}\d v \d w \d u
 $$
 where $\G : (0,\infty)^3 \rightarrow \mathbb R$ is an analytic funtion,
 $s_{j}\in\set{0,\theta_{j}}$ for $j=1,\dots,K$,
 \begin{multline*}
  g^p(u,v,w;s_{1},\dots,s_{K-2})=
   \G \left(\frac{|\T 0w|}{|w|}, \frac{|\T 0v|}{|v|}, \frac{|\T vw|}{|v-w|}\right) 
  \Gamma(u,v,w,s_{1},s_{2})\cdot{}\\
  {}\cdot\left( \bigotimes_{i=3}^{K_1} \dg(u+ s_i v) \right) \otimes \left(\bigotimes_{j=K_1}^{K_2} \dg(u + s_i w) \right) \otimes \left(\bigotimes_{j=K_2}^{K-2} \dg(u+v + s_i (w-v)) \right),
 \end{multline*}
 and
 $\Gamma(u,v,w,s_{1},s_{2})$ is a term of one of the four types
 \begin{align*}
  &\frac{\br{\gamma'(u+s_1 w) - \gamma'(u+s_1 v)}\otimes\br{\gamma'(u+s_2 w) - \gamma'(u+s_2 v)}}{|v|^{p-2} |w|^{p-2} |v-w|^{p}}, \\
  &\frac{|\gamma'(u+s_1 w) - \gamma'(u+s_2 w)|^2}{|v|^{p-2} |w|^{p-2} |v-w|^{p}}, \\ 
  & \frac{|\gamma'(u+s_1 v) - \gamma'(u+s_2 v)|^2}{|v|^{p-2} |w|^{p-2} |v-w|^{p}}, \\ 
  & \frac{|\gamma'(u+v + s_1 (w-v)) - \gamma'(u+v+s_2 (w-v))|^2}{|v|^{p-2} |w|^{p-2} |v-w|^{p}}.
 \end{align*}
\end{lemma}

\begin{proof}
 Using
 \begin{align*}
  \frac{\left|\T w0 \wedge \T v0 \right|^2}{|\T w0|^{p+2}|\T v0|^p |\T vw|^p}  \cdot  \sp{\T w0, \T[h] w0}
  = \G \left(\frac{|\T 0w|}{|w|}, \frac{|\T 0v|}{|v|}, \frac{|\T vw|}{|v-w|}\right)
 \frac{\left|\T w0 \wedge \T v0 \right|^2}{|w|^{p+2}|v|^p |v-w|^p}  \cdot  \sp{\T w0, \T[h] w0} 
 \end{align*}
 where
$
  \G (z_1, z_2, z_2) = z_1^{-p-2} z_2 ^{-p} z_3 ^{-p}
$
together with
 \begin{equation*}
  \frac {\T w0} w \wedge \frac {\T v0} v
  = \left(\frac {\T w0} w- \frac {\T v0} v \right)\wedge \frac {\T v0} v
  = \int_{0}^{1} \int_{0}^1 \left(\gamma'(u + s_1 w) - \gamma'(u+s_1 v) \right) \wedge \gamma'(u+s_3 v) \d s_1 \d s_3
 \end{equation*}
 we see that the first term of $R^1$ is of type 1. Similarly, one gets that all the terms of $R^1$ are
of type 1.

For the term $R^2$ we use
\begin{multline*}
 \sp{P^\bot_{\T 0v} (\T 0w), \T[h] 0w}
 =\frac 1 {|\T 0v|^2} \sp{(\T 0v \wedge \T 0w), (\T 0v \wedge \T[h] 0w)}  \\
 \refeq{abac} \frac {v^{2}w^{2}} {|\T 0v|^2} \idotsint_{[0,1]^4} \gamma'(u+s_1 v) \otimes \gamma'(u+s_2 w) \otimes \gamma'(u+s_3 v) \otimes 
 h'(u+s_4 w) \d s_1 \d s_2 \d s_3 \d s_4
\end{multline*}
together with the fact that for  $w \in \mathbb R$, $u \in \mathbb R / \mathbb  Z$
 \begin{equation} \label{eq:OHaraTypeTerm}
 \begin{split}
  \frac 1{|\T w0|^\alpha} - \frac 1 {|w|^\alpha} &= 2\G[\alpha]\br{\frac{\abs{\T w0}}{|w|}} \frac {1-\frac {\abs{\T w0}^2}{|w|^2}} {|w|^\alpha}\\
  &= \G[\alpha] \br{\frac{\abs{\T w0}}{|w|}} \iint\limits_{[0,1]^2} \frac {|\gamma'(u+s_1 w) - \gamma'(u+s_2 w)|^2}{|w|^\alpha} \d s_1 \d s_2
  \end{split}
  \end{equation}
 where $\G[\alpha] (z) = \tfrac12\cdot\frac {1-z^\alpha}{1-z^2} \cdot z^{-\alpha}$ is analytic on $(0,\infty)$ for $\alpha>0$.
 Both equations together with
\begin{align*}
 \frac 1{\abs{\T vw}^p\abs{\T v0}^{p-2}\abs{\T w0}^p} 
 &- \frac 1{\abs{v-w}^p\abs{v}^{p-2}\abs{w}^p}
\\ &= \left(\frac 1{\abs{\T vw}^p\abs{\T v0}^{p-2}\abs{\T w0}^p} 
 - \frac 1{\abs{\T vw}^p\abs{\T v0}^{p-2}\abs{w}^p} \right)
\\ &+ \left(\frac 1{\abs{\T vw}^p\abs{\T v0}^{p-2}\abs{w}^p} 
 - \frac 1{\abs{\T vw}^p\abs{v}^{p-2}\abs{w}^p} \right)
\\ &+ \left(\frac 1{\abs{\T vw}^p\abs{v}^{p-2}\abs{w}^p} 
 - \frac 1{\abs{v-w}^p\abs{v}^{p-2}\abs{w}^p} \right)
\end{align*}
show that the frist term of $R_2$ is the sum of terms of type 2 to 4. Similarly for the second term
in $R_2$.

Let us turn to the last term, $R_3$.
Again, we restrict to the first term, the second is parallel.
We obtain
\begin{align*}
 \frac{\sp{P^{\top}_{\T v0}
\left( \frac{ \T w0 } w - \frac{ \T v0 } v \right),\frac{\T[h] w0}w}}{\abs{v-w}^p\abs{v}^{p-2}\abs{w}^{p-2}}
&= \frac{\sp{\frac{ \T w0 } w - \frac{ \T v0 } v,\fracabs{\T v0}}}{\abs{v-w}^p\abs{v}^{p-2}\abs{w}^{p-2}}\sp{\fracabs{\T v0} ,\frac{\T[h] w0}w} \\
&= \abs{\frac{\T v0}v}^{-2}\frac{\sp{\frac{ \T w0 } w - \frac{ \T v0 } v,\frac{\T v0}v}}{\abs{v-w}^p\abs{v}^{p-2}\abs{w}^{p-2}} \sp{\frac{\T v0}v ,\frac{\T[h] w0}w}\\
&= \abs{\frac{\T v0}v}^{-2}\frac{\br{\sp{\frac{ \T w0 } w , \frac{ \T v0 } v}-1}}{\abs{v-w}^p\abs{v}^{p-2}\abs{w}^{p-2}} \sp{\frac{\T v0}v ,\frac{\T[h] w0}w} \\
&= -\tfrac12\abs{\frac{\T v0}v}^{-2}\frac{{\abs{\frac{ \T w0 } w - \frac{ \T v0 } v}^{2}}}{\abs{v-w}^p\abs{v}^{p-2}\abs{w}^{p-2}} \sp{\frac{\T v0}v ,\frac{\T[h] w0}w},
\end{align*}
which gives rise to type~4.
\end{proof}

 Our next task is to show that
  $R^p$
 is in fact a lower-order term. More precisely, we have
 
 \begin{proposition}[Regularity of the remainder term]\label{prop:reg-rem}
  If $\g\in\Wia[(3p-4)/2+\s,2]\rzd$ for some $\s\ge0$ then
  $R^p(\g,\cdot)\in\br{\W[3/2-\s+\eps,2]}^{*}$ for any $\eps>0$.
 \end{proposition}
 
 This statement together with Proposition~\ref{prop:Q} immediately leads to
 the proof of the regularity theorem which is deferred to the end of
 this section.
 
 To prove Proposition~\ref{prop:reg-rem}
 , we first note that,
 by partial integration, the terms of $\Re[](\g,h)$ may be transformed into
\begin{align*}
  &\idotsint_{[0,1]^{K}}\iint_{D}\int_{\R/\Z}
   \br{(-\Delta)^{\tilde\s/2}\gp(\cdot,v,w;s_{1},\dots,s_{K-2})}(u) \otimes\br{(-\Delta)^{-\tilde\s/2}\dh}(u+s_{K-1}v+s_{K}w)
   \d u\d v\d w \d\theta_{1}\cdots\d\theta_{K} \\
  &\le \idotsint_{[0,1]^{K}}
  \iint_{D}\norm{\gp (\cdot,v, w;\dots)}_{\W[\tilde\s,1]\rzd}\d v\d w
  \d\theta_{1}\cdots\d\theta_{K-2} \norm{(-\Delta)^{-\tilde\s/2}\dh}_{L^{\infty}\rzd} \\
 &\le \idotsint_{[0,1]^{K}}
  \iint_{D}\norm{\gp (\dots)}_{\W[\tilde\s,1]\rzd}\d v \d w
  \d\theta_{1}\cdots\d\theta_{K-2}\norm{h}_{\W[3/2+\eps/2-\tilde\s,2]\rzd},
 \end{align*}
  where $\tilde\s\in\R$, $\eps>0$ can be chosen arbitrarily, and $(-\Delta)^{\tilde\s/2}$ denotes the fractional
  Laplacian.
  We let $\tilde\s:=0$ if $\s=0$ and $\tilde\s:=\s-\frac\eps2$ otherwise.
  Now the claim directly follows from the succeeding auxiliary result.
 
 \begin{lemma}[Regularity of the remainder integrand]\label{lem:reg-int}
  Let $\g\in\Wia[(3p-4)/2+\s,2]$.
  \begin{itemize}
  \item If $\s=0$ then $\gp\in L^{1}(\R/\Z\times D,\R^{n})$ and
  \item if $\s>0$ then $((v,w) \mapsto \gp(\cdot,v,w;\dots))\in L^{1}(D,\W[\tilde\s,1](\R/\Z, \R^{n}))$ for any $\tilde\s<\s$.
  \end{itemize}
  The respective norms are bounded independently of $s_{1},\dots,s_{K}$.
 \end{lemma}
 
 \begin{proof}
  Recall that the argument of $\G$ is compact and bounded away from zero.
  Using arc-length parametrization, we immediately obtain for the first type
  \begin{align*}
   \iiint\limits_{\mathbb R / \mathbb Z \times D }\abs{\gp(u,v,w)} \d v \d w \d u &\le C \iiint\limits_{\mathbb R / \mathbb Z \times D } 
   \frac{\abs{\dg(u+s_{1}w)-\dg(u+s_{1}v)}^{2}}{ |w|^{p-2} |v|^{p-2} |w-v|^p} \d w \d v \d u \\
   & =  C \iiint\limits_{\R/\Z \times D } \frac{\abs{\dg(u+s_{1}(w-v))-\dg(u)}^{2}}{ |w|^{p-2} |v|^{p-2} |w-v|^p} \d w \d v \d u \\
   &\refeq[\le]{intM<seminorm} C\norm\g_{\W[(3p-4)/2,2]}^{2}.
   \end{align*}
  For a term of the second type we get
  \begin{align*}
   \iiint\limits_{\mathbb R / \mathbb Z \times D } & \abs{\gp(u,v,w)} \d v \d w \d u \le C \iiint\limits_{\mathbb R / \mathbb Z \times D } 
   \frac{\abs{\dg(u+s_{1}w)-\dg(u+s_{2}w)}^{2}}{ |w|^{p-2} |v|^{p-2} |w-v|^p} \d w \d v \d u \\
   & =  C \int_{\R/\Z} \int_0^{2/3} \frac{\abs{\dg(u+s_{1}w)-\dg(u+s_2 w)}^{2}}{ |w|^{p-2}} \left(\int_0^{2/3-w} \frac 1 {v^{p-2}(v+w)^{p}}\d v\right) \d w \d u \\
   & \le C \int_{\R/\Z} \int_0^{2/3} \frac{\abs{\dg(u+s_{1}w)-\dg(u+s_2 w)}^{2}}{ |w|^{3p-5}} \left(\int_{0}^{\infty} \frac 1 {t^{p-2}(1+t)^{p}}\d t\right) \d w \d u \\
   & \le  C\seminorm\g_{\W[(3p-4)/2,2]}
  \end{align*}
  and of course the same estimate is true for a term of the third kind. For term of type four, we get along the same lines
  \begin{align*}
   \iiint\limits_{\mathbb R / \mathbb Z \times D } & \abs{\gp(u,v,w)} \d v \d w \d u \le C \iiint\limits_{\mathbb R / \mathbb Z \times D } 
   \frac{\abs{\dg(u+v+s_{1}(w-v)-\dg(u+v + s_{2}(w-v))}^{2}}{ |w|^{p-2} |v|^{p-2} |w-v|^p} \d w \d v \d u \\
   & =  C \int_{\rzd} \int_0^{2/3} \frac{\abs{\dg(u+s_{1}(w+v))-\dg(u+s_2 (w+v))}^{2}}{ |w|^{p-2}} \left(\int_{0}^{2/3- w} \frac 1 {(v+w)^{p} v^{p-2}}\d v\right) \d w \d u \\
   & \le C \int_{\rzd} \int_0^{2/3} \frac{\abs{\dg(u+s_{1}\tilde w)-\dg(u+s_2 \tilde w)}^{2}}{ |\tilde w|^{3p-5}} \left(\int_{0}^{\infty} \frac 1 {(t +1)^{p}t^{p-2}}\d t\right) \d w \d u \\
   & \le  C\seminorm\g_{\W[(3p-4)/2,2]}.
  \end{align*}
  
  We prove the second claim only for terms of the first type, as the arguments for all other terms follow the same line of argments.
  We will derive a suitable bound on $\norm{\gp(\cdot,w)}_{\W[\tilde\s,r]}$ for some $r>1$.
  To this end, we choose $q_{1},\dots,q_{K-2}$, which will be determined more precisely later on,
  such that
  \begin{equation*}
    \sum_{i=1}^{K-2} \frac 1 {q_i} = \frac 1 r.
  \end{equation*}
  The product rule, Lemma~\ref{lem:product}, then leads to 
  \begin{align*} 
  \norm{\gp(\cdot,w)}_{\W[\tilde\s,r]} &\le C\norm{\G}_{\W[\tilde\s,q_{1}]}\frac{\norm{\dg(\cdot+s_{1}w)-\dg(\cdot+s_{1}v)}_{\W[\tilde\s,2q_{2}]}^{2}}{\abs v ^{p-2}\abs w^{p-2} \abs{w-v}^p}{\prod_{i=3}^{K-2}\norm\dg_{\W[\tilde\s,q_{i}]}}. 
  \\ & =C\norm{\G}_{\W[\tilde\s,q_{1}]}\frac{\norm{\dg(\cdot+s_{1}(w-v))-\dg(\cdot)}_{\W[\tilde\s,2q_{2}]}^{2}}{\abs v ^{p-2}\abs w^{p-2} \abs{w-v}^p}{\prod_{i=3}^{K-2}\norm\dg_{\W[\tilde\s,q_{i}]}}. 
  \end{align*}
  
  For the second factor, we now choose $q_{2}>r$ so small that
  $\W[\s,2]$ embeds into $\W[\tilde\s,2q_{2}]$.
  To this end, we set $\frac1r:=1-(\s-\tilde\s)$ and $\frac1{q_{2}}:=1-2(\s-\tilde\s)$.
  and $q_{i}:=\frac{K-3}{\s-\tilde\s}$ for $i=1,3,4, \ldots, K-2$.
  
  Then for the first factor we apply the chain rule, Lemma~\ref{lem:chain}.
  Recall that $\G$ is analytic and its argument is bounded below away from zero
  and above by~$1$.
  We infer
  \begin{equation*}
   \norm{\G}_{\W[\tilde\s,q_{1}]}
   \leq C \norm\dg_{\W[ \tilde \sigma ,q_{1}]}.
  \end{equation*}
  The Sobolev embedding gives
  \[ \norm{\dg}_{\W[ \tilde \sigma ,q_{i}]} \le C \norm\g_{\W[(3p-4)/2+\s,2]} \le  C \qquad \text{for } i=1,3,4,\dots,K-2. \]
  Together this leads to
  \[ \norm{\gp(\cdot,v,w)}_{\W[\tilde\s,r]} \le C\frac{\norm{\dg(\cdot+s_{1}(w-v))-\dg(\cdot)}_{\W[\s,2]}^{2}}{\abs v^{p-2}\abs w^{p-2} \abs{w-v}^p} \]
  and finally
  \begin{align*}
     \iint_{D} \norm{\gp(\cdot,w)}_{\W[\tilde\s,r]} \d v \d w
     &\le C \iint_{D} \frac{\norm{\dg(\cdot+s_{1}(w-v))-\dg(\cdot)}_{\W[\s,2]}^{2}}{\abs v^{p-2}\abs w^{p-2} \abs{w-v}^p} \d v \d w \\
     &\refeq[\le]{intM<seminorm} C\norm\g_{{\W[(3p-4)/2,2]}}.
  \end{align*} 
 \end{proof}
 
 \begin{proof}[Theorem~\ref{thm:smooth}]
  We start with the Euler-Lagrange Equation
  \begin{equation}\label{eq:elg}
   \delta\E(\g,h)+\lambda\sp{\dg,\dh}_{L^2} = 0
  \end{equation}
  for any $h\in C^\infty\rzd$
  where $\lambda\in\R$ is a Lagrange parameter stemming from the side condition (fixed length).
  Using~\eqref{eq:Re} this reads
  \begin{equation}\label{eq:elg2}
   12\Q(\g,h) + \lambda\sp{\dg,\dh}_{L^2} + 12\Re[](\g,h) = 0.
  \end{equation}
  Since first variation of the length functional satisfies
  \begin{equation*}
   \sp{\dg,\dh}_{L^2} = \sum_{k\in \mathbb Z} |2\pi k|^2 \sp{\hat\g_{k},\hat h_{k}}_{\C^{d}},
  \end{equation*}
  we get using  Proposition~\ref{prop:Q} that there is a $\tilde c>0$ such that
  \begin{equation}\label{eq:LinearTerm}
  12\Q(\g,h) + \lambda\sp{\dg,\dh}_{L^2} = \sum_{k\in\Z} \tilde\rho_k\sp{\hat \g_k,\hat h_k}_{\C^d}
  \end{equation}
  where
  \[ \tilde\rho_k = \tilde c\abs k^{3p-4} + o\br{\abs k^{p-1}}
     \qquad\text{as }\abs k\nearrow\infty. \]
  Assuming that $\g\in\Wia[(3p-4)/2+\s,2]$ for some $\s\ge0$,
  we infer
  \[ 12\Q(\g,\cdot) + \lambda\sp{\dg,\cdot'}_{L^2} \in \br{\W[3/2-\s+\eps,2]}^{*} \]
   applying Proposition~\ref{prop:reg-rem} to~\eqref{eq:elg2}.
  Equation~\eqref{eq:LinearTerm} implies
  \[ \br{\tilde\rho_{k}\abs k^{-3/2+\sigma-\eps}\hat\g_{k}}_{k\in\Z} \in \ell^{2}. \]
  Recalling that $\tilde\rho_{k}\abs k^{{-3p+4}}$ converges to a positive constant as $\abs k\nearrow\infty$, we are led to
  \[ \g\in\W[\textstyle\frac{3p-4}2+\s+\frac{3p-7}2-\eps]. \]
  Choosing $\eps:=\frac{3p-7}4>0$, we gain a positive amount of regularity
  that does not depend on~$\s$. So by a simple iteration we get $\g \in W^{s,2}$
   for all $s \geq 0$.
  \end{proof}

%% file: fractional.tex

\section{Product and chain rule}\label{sect:fractional}

As in~\cite{blatt-reiter2}, we make use of the following results which
we briefly state for the readers' convenience.

\begin{lemma}[Product rule]\label{lem:product}
  Let $q_1,\dots,q_{k} \in (1,\infty)$ with $\sum_{i=1}^{k}\frac1{q_{k}}=\frac1r\in(1,\infty)$ and 
  $s > 0$.
  Then, for $f_{i}\in\W[s,q_{i}]\rzd$, $i=1,\dots,k$,
 \begin{equation*}
  \norm{\prod_{i=1}^{k} f_{i}}_{\W[s,r]} \leq C_{k,s} \prod_{i=1}^{k}\norm{f_{i}}_{\W[s,q_{i}]}.
 \end{equation*}
\end{lemma}

We also refer to Runst and Sickel~{\cite[Lem.~5.3.7/1~(i)]{RS}}. ---
For the following statement, one mainly has to treat $\norm{(D^k\psi)\circ f}_{\W[\s,p]}$ for $k\in\N\cup\set0$ and $\s\in(0,1)$
which is e.g.\@ covered by~\cite[Thm.~5.3.6/1~(i)]{RS}.

\begin{lemma}[Chain rule]\label{lem:chain}
 Let $f\in\W[s,p](\mathbb R / \mathbb Z, \mathbb R^n)$, $s > 0$, $p \in (1,\infty)$.
 If $\psi\in C^\infty(\R)$ is globally Lipschitz continuous and $\psi$ and all its derivatives vanish at~$0$ then $\psi\circ f\in\W[s,p]$ and
 \[ \|\psi \circ f\|_{\W[s,p]} \le C\|\psi\|_{C^{\scriptstyle k}}\|f\|_{\W[s,p]} \]
 where $k$ is the smallest integer greater than or equal to $s$.
\end{lemma}

\section{Equivalence of fractional seminorms}\label{sect:equiv}

We give a straightforward
proof of the equivalence of two seminorms on the Sobolev-Slobodecki\u{\i} spaces we used in this article.

\begin{lemma}\label{lem:equivalenceOfNorms}
 For $s \in (0,1)$, $p\in[1,\infty)$ the seminorms
 \begin{align}\tag{\ref{eq:Wsemi}}
 \seminorm{f}_{\W[1+s,p]} &:= \br{\int_{{\R/\Z}}\int_{-1/2}^{1/2}
 \frac{\abs{f'(u+w)-f'(u)}^p}{\abs w^{1+sp}} \d w\d u}^{1/p}, \\
 \tag{\ref{eq:Wsemi'}}
 \seminormv{f}_{\W[1+s,p]} &:= \br{\int_{{\R/\Z}}\int_{-1/4}^{1/4}
 \frac{\abs{f(u+w)-2f(u)+f(u-w)}^p}{\abs w^{1+(1+s)p}} \d w\d u}^{1/p}
 \end{align}
 are equivalent on $W^{1,p}$.
\end{lemma}

\begin{proof}
 We first prove the equivalence of the two norms for smooth $f$.
 The fundamental theorem of calculus and the triangle inequality tell us
 \begin{align*}
  \seminormv{f}_{\W[1+s,p]}
  &=
  \left(\int_{\mathbb R / \mathbb Z} \int_{-1/4}^{1/4} 
    \frac{|f(u+w)-2f(u)+f(u-w)|^p}{|w|^{1+(s+1)p}} \d w\d u \right)^{1/p} \\
    &  = \left(  \int_{\mathbb R / \mathbb Z} \int_{-1/4}^{1/4} 
    \frac{|\int_0^1 f'(u+\tau w) - f'(u-\tau w) \d\tau|^p}{|w|^{1+sp}} \d w\d u
    \right)^{1/p}
    \\
    &  \leq \left(  \int_{\mathbb R / \mathbb Z} \int_{-1/4}^{1/4} 
    \frac{|\int_0^1 f'(u+\tau w) - f'(u) \d\tau|^p}{|w|^{1+sp}} \d w\d u
    \right)^{1/p}
    \\
   &  \quad + \left(  \int_{\mathbb R / \mathbb Z} \int_{-1/4}^{1/4} 
    \frac{|\int_0^1 f'(u) - f'(u-\tau w) \d\tau|^p}{|w|^{1+sp}} \d w\d u
    \right)^{1/p}
    \\
    &  = 2 \left(  \int_{\mathbb R / \mathbb Z} \int_{-1/4}^{1/4} \int_0^1 
    \frac{|f'(u+\tau w) - f'(u) |^p}{|w|^{1+sp}} \d\tau \d w\d u
    \right)^{1/p}.
 \end{align*}
Using Fubini's theorem and substituting $\tilde w = \tau w$, we can estimate this 
further by
\begin{align*}
 &2 \left(  \int_{\mathbb R / \mathbb Z} \int_0^1 \tau^{sp}\int_{-\tau/4}^{\tau/4} 
    \frac{|f'(u+\tilde w) - f'(u) |^p}{|\tilde w|^{1+sp}} \d\tilde w \d\tau \d u
    \right)^{1/p} \leq \frac {2}{1+sp} \seminorm f_{W^{1+s,p}}.
\end{align*}
Hence,
\begin{equation*}
 \seminormv f_{W^{1+s,p}} \leq \frac {2}{1+sp} \seminorm f_{W^{1+s,p}}.
\end{equation*}
To get an estimate in the other direction, we calculate
for $\varepsilon >0$
\begin{align*}
\seminormv f_{\W[1+s,p]}
 &=\left(\int_{\mathbb R / \mathbb Z} \int_{-1/4}^{1/4} 
    \frac{|f(u+w)-2f(u)+f(u-w)|^p}{|w|^{1+(s+1)p}} \d w\d u \right)^{1/p} \\
    &  = \left(  \int_{\mathbb R / \mathbb Z} \int_{-1/4}^{1/4} 
    \frac{|\int_0^1 f'(u+\tau w) - f'(u-\tau w) \d\tau|^p}{|w|^{1+sp}} \d w\d u
    \right)^{1/p} \\
    &   \geq \left(  \int_{\mathbb R / \mathbb Z} \int_{-1/4}^{1/4} 
    \frac{|\int_{1-\varepsilon}^1 f'(u+\tau w) - f'(u-\tau w) \d\tau|^p}{|w|^{1+sp}} \d w\d u
    \right)^{1/p}
    \\ & \qquad{}-  \left(  \int_{\mathbb R / \mathbb Z} \int_{-1/4}^{1/4}
    \frac{|\int_0^{1-\varepsilon} f'(u+\tau w) - f'(u-\tau w) \d\tau|^p}{|w|^{1+sp}} \d w\d u
    \right)^{1/p} \\
   & =: I_1 - I_2.
\end{align*}
Substituting $\tilde \tau = \frac {\tau}{1-\varepsilon}$
and $\tilde w = (1-\varepsilon)w$, we get
\begin{align*}
 I_2 &= \left(  \int_{\mathbb R / \mathbb Z} (1-\varepsilon)^{(1+s)p}
    \int_{-(1-\varepsilon)/4}^{(1-\varepsilon)/4}
    \frac{|\int_{0}^1 f'(u+\tilde \tau \tilde w) - f'(u-\tilde \tau \tilde w) \d\tilde \tau|^p}{|\tilde w|^{1+sp}} 
    \d\tilde w \d u
    \right)^{1/p}
    \\
   & \leq (1-\varepsilon)^{1+s} \seminormv f_{W^{1+s,p}}.
\end{align*}
For $I_1$ we observe
\begin{align*}
 I_1 &\geq \left(  \int_{\mathbb R / \mathbb Z} \int_{-1/4}^{1/4}
    \frac{|\int_{1-\varepsilon}^1 f'(u+ w) - f'(u- w) \d\tau|^p}{|w|^{1+sp}} \d w\d u
    \right)^{1/p} \\
    & \quad
    {}- \left(  \int_{\mathbb R / \mathbb Z} \int_{-1/4}^{1/4} 
    \frac{|\int_{1-\varepsilon}^1 f'(u+ \tau w) - f'(u+w) \d\tau|^p}{|w|^{1+sp}} \d w\d u
    \right)^{1/p} \\
    & \quad {}- \left(  \int_{\mathbb R / \mathbb Z} \int_{-1/4}^{1/4} 
    \frac{|\int_{1-\varepsilon}^1 f'(u-\tau w) - f'(u- w) \d\tau|^p}{|w|^{1+sp}} \d w\d u
    \right)^{1/p} \\
    &  = \varepsilon \left(  \int_{\mathbb R / \mathbb Z} \int_{-1/4}^{1/4} 
    \frac{|f'(u+ w) - f'(u- w) |^p}{|w|^{1+sp}} \d w\d u
    \right)^{1/p} \\
    & \quad
    {}- 2\left(  \int_{\mathbb R / \mathbb Z} \int_{-1/4}^{1/4} 
    \frac{|\int_{1-\varepsilon}^1 f'(u+ \tau w) - f'(u+w) \d\tau|^p}{|w|^{1+sp}} \d w\d u
    \right)^{1/p}.
\end{align*}
To bound the first integral from below, we calculate
\begin{align*}
 \seminorm f_{\W[1+s,p]} & = \left(  \int_{\mathbb R / \mathbb Z} \int_{-1/2}^{1/2} 
    \frac{|f'(u+  w) - f'(u) |^p}{|w|^{1+sp}} \d w\d u
    \right)^{1/p}\\
 &= 2^{-s} \left(  \int_{\mathbb R / \mathbb Z} \int_{-1/4}^{1/4} 
    \frac{|f'(u+ w) - f'(u- w) |^p}{|w|^{1+sp}} \d w\d u
    \right)^{1/p} .
\end{align*}
The second integral can be estimated further
\begin{align*}
   &\int_{\mathbb R / \mathbb Z} \int_{-1/4}^{1/4} 
    \frac{|\int_{1-\varepsilon}^1 f'(u+ \tau w) - f'(u+w) \d\tau|^p}{|w|^{1+sp}} \d w\d u
   \\
   &
   \qquad\leq \varepsilon^{p-1} \int_{\mathbb R / \mathbb Z} \int_{-1/4}^{1/4} \int_{1-\varepsilon}^1
    \frac{| f'(u+ \tau w) - f'(u+w) |^p}{|w|^{1+sp}} \d\tau \d w\d u
   \\
   & 
   \qquad = \varepsilon^{p-1} \int_{\mathbb R / \mathbb Z} \int_{-1/4}^{1/4} \int_{1-\varepsilon}^1
    \frac{| f'(u+ (\tau-1) w) - f'(u) |^p}{|w|^{1+sp}} \d\tau \d w\d u
   \\
   &
   \qquad =  \varepsilon^{p-1}\int_{\mathbb R / \mathbb Z} \int_{-1/4}^{1/4} \int_{0}^{\varepsilon}
    \frac{| f'(u- \tau w) - f'(u) |^p}{|w|^{1+sp}} \d\tau \d w\d u
   \\
   &\qquad\leq \frac {\varepsilon^{sp+p}} {sp+1} \int_{\mathbb R / \mathbb Z} \int_{-1/4}^{1/4}
    \frac{| f'(u+ w) - f'(u) |^p}{|w|^{1+sp}} \d w\d u.
\end{align*}
So we finally arrive at
\begin{align*}
 \seminormv f_{W^{1+s,p}} \geq
 2^{s} \varepsilon \seminorm f_{\W[1+s,p]}
   -2\frac {\varepsilon^{s+1}}{\sqrt[p]{sp+1}} \seminorm f_{\W[1+s,p]}
-(1-\varepsilon)^{1+s} \seminormv f_{W^{1+s,p}}.
\end{align*}
For $\varepsilon = 2^{1-2/s}$ this leads to
\begin{equation*}
 \seminormv f_{W^{1+s,p}} \geq 2^{-1-2/s} \seminorm f_{W^{1+s,p}}.
\end{equation*}

To get the statement for $f\in W^{1,p}$ we use a standard mollifier $\phi \in C^{\infty}(\mathbb R)$
with $\phi\ge0$, bounded support and
$$
 \int_{\mathbb R} \phi \d x = 1.
$$
We set $\phi_{\varepsilon} (x):= \frac 1 \varepsilon \phi(\frac x \varepsilon)$ and
$$
  f_\varepsilon = f \ast \phi_{\varepsilon}.
$$
Then $f_\varepsilon$ converges to $f$ in $W^{1,p}$ and we can hence chose a sequence $\varepsilon_k \rightarrow 0$
such that
$
  f_k := f_{\varepsilon_k}
$
converge pointwise almost everywhere to $f$ and $f_k'$ to $f'$.

Using H\"older's inequality, we see that
\begin{align*}
 \seminorm{f_\varepsilon}_{W^{s+1,p}}^p
     &=  \int_{\mathbb R / \mathbb Z} \int_{-1/2}^{1/2} \frac{|\int_{\mathbb R}(f'(u+w-z)-f'(u-z))  \phi_\varepsilon(z)\d z|^p}{|w|^{1+sp}} \d w\d u
  \\ &\leq \br{\int_{\R}\phi_{\eps}(\tilde z)\d\tilde z}^{p-1}\int_{\mathbb R}\int_{\mathbb R / \mathbb Z} \int_{-1/2}^{1/2} \frac{|(f'(u+w-z)-f'(u-z)) |^p\phi_\varepsilon(z)}{|w|^{1+sp}} \d w\d u\d z
  \\ & = \int_{\mathbb R}\phi_{\eps}( z)\int_{\mathbb R / \mathbb Z} \int_{-1/2}^{1/2} \frac{|(f'(u+w)-f'(u)) |^p}{|w|^{1+sp}} \d w\d u \d z
  \\ & =\seminorm f _{W^{s+1,p}} ^{p}
\end{align*}
for all $\varepsilon >0$. Similarly,
$$
  \seminormv{f_\varepsilon} _{W^{s+1,p}}
  \leq\seminormv f_{W^{s+1,p}}.
$$

Hence Fatou's lemma tells us that
\begin{equation*}
 \seminormv f_{W^{1+s,p}} \leq \liminf_{k \rightarrow \infty} \seminormv{f_k}_{W^{1+s,p}} \leq C \liminf_{k \to\infty} \seminorm{f_k}_{W^{1+s,p}}
 \leq C\seminorm f_{W^{1+s,p}}
\end{equation*}
and
\begin{equation*}
 \seminorm f_{W^{1+s,p}} \leq \liminf_{k \rightarrow \infty}
 \seminorm{f_k}_{W^{1+s,p}} \leq C \liminf_{k \to \infty} \seminormv{f_k}_{W^{1+s,p}} 
 \leq C\seminormv f_{W^{1+s,p}} .
\end{equation*}
This completes the proof of the lemma.
\end{proof}